\documentclass[11pt]{article}  

\usepackage{amssymb,amsmath,amsthm}
\usepackage[alphabetic]{amsrefs}
\usepackage{color}
\usepackage{hyperref}

\usepackage[a4paper, left=20mm, right=20mm, top=30mm, bottom=25mm]{geometry}

\newtheorem{theo}{Theorem}[section]
\newtheorem{theoremalpha}{Theorem} 

\newtheorem{lemm}[theo]{Lemma}
\newtheorem{corr}[theo]{Corollary}
\newtheorem{prop}[theo]{Proposition}

\numberwithin{equation}{section}

\theoremstyle{definition}
\newtheorem{defi}[theo]{Definition}
\newtheorem{exam}[theo]{Example}
\newtheorem{rema}[theo]{Remark}

\newcommand{\C}{\mathbb{C}}
\newcommand{\R}{\mathbb{R}}

\newcommand{\Z}{\mathbb{Z}}

\newcommand{\Sym}{\operatorname{Sym}}

\newcommand{\cone}[1]{\overline{#1}}

\newcommand{\Ad}{\mathrm{Ad}}
\newcommand{\sphere}[1]{\mathbb{S}^{#1}}
\newcommand{\hb}[1]{\mathbb{H}^{#1}}

\DeclareMathOperator{\sgn}{\mathrm{sgn}}
\DeclareMathOperator{\Ric}{\mathrm{Ric}}
\DeclareMathOperator{\scal}{\mathrm{Scal}}
\DeclareMathOperator{\rank}{\mathrm{rank}}
\DeclareMathOperator{\End}{\mathrm{End}}
\DeclareMathOperator{\Hom}{\mathrm{Hom}}
\DeclareMathOperator{\Id}{\mathrm{Id}}
\DeclareMathOperator{\sldirac}{{\partial\!\!\!/}}

\newcommand{\Spin}{\mathrm{Spin}}
\newcommand{\SO}{\mathrm{SO}}
\newcommand{\SU}{\mathrm{SU}}

\newcommand{\g}{\mathfrak{g}}
\newcommand{\h}{\mathfrak{h}}
\newcommand{\m}{\mathfrak{m}}
\renewcommand{\sl}{\mathfrak{sl}}
\newcommand{\su}{\mathfrak{su}}
\newcommand{\so}{\mathfrak{so}}
\newcommand{\spin}{\mathfrak{spin}}
\renewcommand{\Re}{\operatorname{Re}}

\title{\textbf{Higher spin Killing spinors on 3-dimensional manifolds}}
\date{\today}
\author{Yasushi Homma, Natsuki Imada, and Soma Ohno}

\begin{document}
\maketitle

\begin{abstract}
We define higher spin Killing spinors on Riemannian spin manifolds in arbitrary dimension and study them in detail in dimension three.
We prove a rigidity result for $3$-dimensional manifolds admitting higher spin Killing spinors and give expressions for higher spin Killing spinors on the 3-sphere and the 3-hyperbolic space explicitly. We also investigate the Killing spinor type equation on integral spin bundles.
\end{abstract}


\section{Introduction}
Killing spinors are special sections of the spinor bundle on spin manifolds, which are defined by a first-order differential equation $\nabla_X \varphi = \mu X \cdot \varphi$ for any vector field $X$ and a constant $\mu \in \C$. In particular, when $\mu$ is a non-zero real number, $\varphi$ is called a real Killing spinor.
They have been studied extensively in mathematics and physics. From a mathematical viewpoint, Killing spinors appear naturally in the study of the eigenvalue problem of the Dirac operator. Indeed, Friedrich's inequality for the first eigenvalue of the Dirac operator is well-known and the equality holds if and only if the manifold admits a Killing spinor with real Killing number \cite{BFGK}. 
Manifolds admitting Killing spinors have some special geometric structures. For example, such manifolds must be Einstein. 
Moreover, the classification of manifolds admitting real Killing spinors was obtained by C. B\"ar \cite{BarKilling}. 
According to his results, a simply connected complete spin manifold admitting real Killing spinors is homothetic to one of the following manifolds: the sphere, an Einstein-Sasakian manifold, a 3-Sasakian manifold, a nearly K\"ahler manifold in dimension 6, or a nearly parallel $\mathrm{G}_2$-manifold in dimension 7.
From a physical viewpoint, Killing spinors play important roles in supergravity theory. Killing spinors represent the parameters of preserved supersymmetries of a solution of a physical system \cite{SUGRA,KKSUGRA2025}.

In usual spin geometry, spinor fields are sections of the spinor bundle associated with the spin $\frac{1}{2}$ representation of the spin group.
Recently, spin geometry with higher spin representations has been studied actively \cite{UweHomma, HT, OT23, spin3/2SW, Richtsfeld24}. 
One of the motivations to study higher spin geometry comes from physics.
In 1941, Rarita and Schwinger \cite{RS41} proposed a field equation for particles with arbitrary half-integral spin, which is now called the Rarita-Schwinger equation.

In this paper, we generalize the notion of Killing spinors to spinor fields with higher spin.
After reviewing spin geometry with higher spin, we define the higher spin Killing spinors in \S 3.
We also study some basic properties of higher spin Killing spinors.
In the spin $\frac{1}{2}$ case, Killing spinors are closely related to Killing vector fields.
Indeed, we can construct Killing vector fields from Killing spinors. 
This is a reason why they are called Killing spinors.
We generalize this relation to the higher spin cases in Proposition \ref{prop:higher_Killing_spinor_and_Killingtensor}.
However, in dimension $\geq 4$, it seems difficult to construct non-trivial examples of higher spin Killing spinors. In fact, we could not find any non-trivial examples with non-zero Killing number in dimension $\geq 4$.

So, in \S \ref{sec:3dimcase}, we focus on 3-dimensional manifolds and study the properties of higher spin Killing spinors in detail.
In particular, we generalize three well-known results of usual Killing spinors to the higher spin cases.
First, we give the following rigidity result for $3$-dimensional manifolds admitting higher spin Killing spinors in \S \ref{sec:higher_spin_killing_on_3mfd}.

\newtheorem*{MainTheorem1}{\rm\bf Theorem~\ref{theo:3dim_Killing_and_Einstein}}
\begin{MainTheorem1}
  Let $(M,g)$ be a 3-dimensional spin manifold. If $M$ admits a spin $(j + \frac{1}{2})$ Killing spinor $\varphi$ with Killing number $\mu$, then $(M,g)$ is an Einstein manifold, and hence, $M$ is of constant curvature with $\scal = 24\mu^2$.
\end{MainTheorem1}
This theorem is a reason why we study higher spin Killing spinors on the 3-sphere in \S \ref{sec:higher_killing_on_S3} and the 3-hyperbolic space in \S \ref{sec:higher_killing_on_H3}.
Second, we prove the cone construction for higher spin Killing spinors in \S \ref{sec:cone_construction}. C. B\"ar \cite{BarKilling} proved that there is a one-to-one correspondence between real Killing spinors on a Riemannian spin manifold and parallel spinors on its cone. We generalize this result to the higher spin cases as follows.
\newtheorem*{MainTheorem2}{\rm\bf Theorem~\ref{theo:cone_construction}}
\begin{MainTheorem2}
  Let $(M,g)$ be a 3-dimensional Riemannian spin manifold and $(\cone{M}, \cone{g})$ the cone over $(M,g)$. Then, the following two are equivalent:
  \begin{enumerate}
    \item $M$ admits a spin $(j + \frac{1}{2})$ Killing spinor with Killing number $\mu = \frac{1}{2}$ (resp. $\mu = -\frac{1}{2}$).
    \item The cone $\cone{M}$ admits a parallel spinor on the bundle $\cone{S}_{j,0}$ with helicity $j + \tfrac{1}{2}$ (resp. $\cone{S}_{0,j}$ with $-(j + \tfrac{1}{2})$).
  \end{enumerate}
\end{MainTheorem2}
Third, we give an eigenvalue estimate for the higher spin Dirac operator on compact 3-dimensional manifolds and discuss a relation between the estimate and higher spin Killing spinors in \S \ref{sec:eigenvalue_estimate}.
\newtheorem*{MainTheorem3}{\rm\bf Theorem~\ref{theo:eigenvalue_estimate}}
\begin{MainTheorem3}
  Let $(M,g)$ be a compact 3-dimensional Riemannian spin manifold and 
  \begin{align*}
    r_0 := \min_M q(R) :=& \min \{d \in \R \mid x \in M, d \text{ is an eigenvalue of } q(R)_x \}\\
    =& \max \{d \in \R \mid \forall x \in M, q(R)_x \geq d\}
  \end{align*}
  Then, the first eigenvalue $\lambda^2$ of $D_j^2$ on $\ker T_j^-$ satisfies
  \[\lambda^2 \geq \frac{2j+3}{2j+1}r_0.\]
  The equality holds if and only if $M$ admits a spin $(j + \frac{1}{2})$ Killing spinor.
\end{MainTheorem3}

In \S\ref{sec:higher_killing_on_S3} and \S\ref{sec:higher_killing_on_H3}, we give examples of higher spin Killing spinors on 3-dimensional manifolds of constant curvature. 
In \S \ref{sec:higher_killing_on_S3}, we give a construction of higher spin Killing spinors on the 3-sphere $\sphere{3}$. As a consequence, we know that spin $(j + \tfrac{1}{2})$ Killing spinors on $\sphere{3}$ are obtained from spin $(j - \tfrac{1}{2})$ Killing spinors on $\sphere{3}$ inductively.
In \S \ref{sec:higher_killing_on_H3}, we study higher spin Killing spinors on the 3-dimensional hyperbolic space $\hb{3}$. Consequently, we obtain the explicit expressions of higher spin Killing spinors on $\hb{3}$.

The final section \S\ref{sec:integral_spin} is devoted to the study of a Killing spinor-type equation on the integral spin bundles, which is realized as the bundle of traceless symmetric tensors. 
In \S \ref{sec:Killing_spinor_type_eq_on_int_spin}, we see that solutions to this equation form a special class of traceless Killing tensors.
In \S \ref{sec:Killing_tensor_on_S3}, we focus on the case of $\sphere{3}$ and study a relation between higher spin Killing spinors and Killing tensors.

\section{Spin geometry with higher spin}
\label{sec:preliminary}
Let $(M, g)$ be an $n$-dimensional $(n \geq 3)$ Riemannian spin manifold with a spin structure $P_{\Spin}M$, which is a principal $\Spin(n)$-bundle over $M$ and a double cover of the orthonormal frame bundle $P_{\SO}M$.
We consider the spin $(j + \frac{1}{2})$ unitary representation $\pi_j$ on $W_j$ of $\Spin(n)$ for $j \in \Z_{\geq 0}$. 
Here, the spin $(j + \frac{1}{2})$ representation can be described by its highest weight:
\[W_j = \begin{cases}
  L(j + \tfrac{1}{2}, (\tfrac{1}{2})_{m-1}) & \text{for } n = 2m + 1, \\
  L(j + \tfrac{1}{2}, (\tfrac{1}{2})_{m-2}, \tfrac{1}{2}) \oplus L(j + \tfrac{1}{2}, (\tfrac{1}{2})_{m-2}, -\tfrac{1}{2}) & \text{for } n = 2m.
\end{cases}\]
In the above, $L(a_1, a_2, \ldots, a_k)$ denotes the irreducible representation space of $\Spin(n)$ whose highest weight is $(a_1, a_2, \ldots, a_k)$ and $(\tfrac{1}{2})_k$ denotes a sequence $\tfrac{1}{2}, \ldots, \tfrac{1}{2}$ with length $k$ as an abbreviation.
For example, the spin $\frac{1}{2}$ representation $(\pi_0, W_0)$ is a usual spinor representation. 

The representation $\pi_j$ induces a vector bundle $S_j$ associated with the principal bundle $P_{\Spin} M$. Indeed, we consider the action of $\Spin(n)$ on $P_{\Spin} M \times W_j$ by
\[
\Spin(n) \times (P_{\Spin} M \times W_j) \ni (g, (p, v)) \mapsto (pg^{-1}, \pi_j(g)v) \in P_{\Spin} M \times W_j.
\]
Then, we have a Hermitian vector bundle whose fiber is $W_j$,
\[
S_j := P_{\Spin} M \times_{\pi_j} W_j = (P_{\Spin} M \times W_j)/\Spin(n).
\]
Since $(\pi_0, W_0)$ is a usual spinor representation, $S_0$ is a usual spinor bundle of $M$.
We call a section of $S_j$ a spin $(j + \frac{1}{2})$ field or a spinor field with spin $(j + \frac{1}{2})$.

From now on, we study some first-order differential operators naturally defined on the space of the spin $(j + \frac{1}{2})$ fields $\Gamma(S_j)$.
They are called generalized gradients or Stein–Weiss operators, which are defined by composing the orthogonal bundle projections and the connection $\nabla$ on $S_j$ induced by the Levi-Civita connection.
  
The connection $\nabla$ can be seen as a map
\[
\nabla \colon \Gamma(S_j) \ni \varphi \mapsto \nabla\varphi = \sum \nabla_{e_i} \varphi \otimes e_i \in \Gamma(S_j \otimes TM^{\mathbb{C}}),
\]
where $\{e_i\}_i$ is a local orthonormal frame and $TM^{\mathbb{C}}$ is $TM \otimes \mathbb{C} \simeq T^*M \otimes \mathbb{C}$ by the Riemannian metric $g$. 
The fiber of the vector bundle $S_j \otimes TM^{\mathbb{C}}$ is $W_j \otimes \mathbb{C}^n$, and it is decomposed into 4 components as a representation of $\Spin(n)$:
\[
W_j \otimes \mathbb{C}^n \cong W_{j+1} \oplus W_j \oplus W_{j-1} \oplus W_{j,1},
\]
where $W_{j,1}$ is a $\Spin(n)$-module defined by
\[W_{j,1} =
\begin{cases}
  L(j + \tfrac{1}{2}, \tfrac{3}{2}, (\tfrac{1}{2})_{m-2}) & \text{for } n = 2m+1\\
  L(j + \tfrac{1}{2}, \tfrac{3}{2}, (\tfrac{1}{2})_{m-3}, \tfrac{1}{2}) \oplus L(j + \tfrac{1}{2}, \tfrac{3}{2}, (\tfrac{1}{2})_{m-3}, -\tfrac{1}{2}) & \text{for } n = 2m.
\end{cases}
\]
We note that $W_{j,1}$ does not appear for $n = 3$ or $j = 0$, nor does $W_{j-1}$ for $j=0$. We will study the 3-dimensional case in more detail in \S \ref{sec:3dimcase}.

The space $W_j$ has a $\Spin(n)$-invariant Hermitian inner product (unique up to a constant factor), so that the above decomposition is orthogonal. So, we have orthogonal decomposition of the bundle $S_j \otimes TM^{\C} \cong S_{j+1} \oplus S_j \oplus S_{j-1} \oplus S_{j,1}$, where $S_{j,1}$ is an associated vector bundle with the fiber $W_{j,1}$. 
Let $\Pi_j \colon S_j \otimes TM^{\mathbb{C}} \to S_j$ be the orthogonal projection onto the $S_j$-component.
Composing the connection $\nabla$ and the projection $\Pi_j$, we obtain the so-called \emph{higher spin Dirac operator},
\[
D_j := \Pi_j \circ \nabla \colon \Gamma(S_j) \to \Gamma(S_j \otimes TM^{\mathbb{C}}) \to \Gamma(S_j).
\]
In this manner, we construct four generalized gradients on $\Gamma(S_j)$ and name them as follows:
\begin{align*}
  D_j &\colon \Gamma(S_j) \to \Gamma(S_j) &&\text{the higher spin Dirac operator}, \\
  T^+_j &\colon \Gamma(S_j) \to \Gamma(S_{j+1}) &&\text{the (first) twistor operator}, \\
  T^-_j &\colon \Gamma(S_j) \to \Gamma(S_{j-1}) &&\text{the co-twistor operator},\\
  U_j &\colon \Gamma(S_j) \to \Gamma(S_{j,1}) &&\text{the (second) twistor operator}. 
\end{align*}
Here, we set $U_j = 0$ for $n = 3$, $U_0 = 0$, and $T^-_0 = 0$.

Similarly to the generalized gradients, we can define the action of tangent vectors on spinor fields. For a tangent vector $X \in T_x M\,\, (\forall x \in M)$, we define the \emph{Clifford homomorphism} $p_j(X) \colon (S_j)_x \to (S_j)_x$ by
\[p_j(X)\varphi := \Pi_j(\varphi \otimes X), \quad \text{for } \varphi \in (S_j)_x.\]
We also define other Clifford homomorphisms in the same manner:
\begin{align*}
  p_j(X) \colon (S_j)_x &\to (S_j)_x, & p^+_j(X) \colon (S_j)_x &\to (S_{j+1})_x, \\
  p_{j,1}(X) \colon (S_j)_x &\to (S_{j,1})_x, & p^-_j(X) \colon (S_j)_x &\to (S_{j-1})_x.
\end{align*}
Here, we set $p_{j,1}(X) = 0$ for $n = 3$, and $p_{0,1}(X) = 0$, and $p^-_{0}(X) = 0$.
We can locally express the generalized gradients using the Clifford homomorphisms as follows:
\begin{align*}
  D_j &= \sum_i p_j(e_i) \nabla_{e_i}, & T^+_j &= \sum_i p^+_j(e_i) \nabla_{e_i},\\
  U_j &= \sum_i p_{j,1}(e_i) \nabla_{e_i}, & T^-_j &= \sum_i p^-_j(e_i) \nabla_{e_i},
\end{align*}
where $\{e_i\}_i$ is a local orthonormal frame of $TM$.

\begin{rema}
  \label{rema:twisted_Dirac}
  These operators can be realized as a decomposition of a twisted Dirac operator. 
  We consider the vector bundle $S_0 \otimes \Sym^j_0$ and the twisted Dirac operator on this vector bundle defined by
  \[\sldirac = \sum_{i=1}^n (e_i \cdot \otimes \Id_{\Sym^j_0}) \circ \nabla_{e_i}.\]
  Here, $\Sym^j_0 := \Sym^j_0(TM^{\mathbb{C}})$ is the traceless $j$-th symmetric tensor product of complexified tangent bundle, and $\nabla$ is the connection on $S_0 \otimes \Sym^j_0$ induced from the Levi-Civita connection.
  Since the fiber of $\Sym^j_0$ is $L(j, 0_{m-1})$, the tensor bundle is decomposed as $S_0 \otimes \Sym^j_0 \cong S_j \oplus S_{j-1}$. 
  Then, the twisted Dirac operator $\sldirac$ is decomposed into the following $2 \times 2$ matrix form:
  \[\sldirac = \begin{pmatrix}
       a D_j & b T^+_{j-1} \\
      c T^-_j & d D_{j-1}
    \end{pmatrix},\]
  where $a, b, c, d$ are some non-zero constants depending on $n$ and $j$.
\end{rema}

At the end of this section, we shall show some relations among the Clifford homomorphisms, which are generalizations of the well-known relation for the Clifford multiplication.
\begin{prop}[{\cite{HT}}]
  \label{prop:relative_dimension_formula}
  Let $(\rho, W_{\rho})$ be an irreducible representation of $\Spin(n)$ and $W_{\rho} \otimes \C^n \cong \bigoplus_{\lambda} W_{\lambda}$ be the irreducible decomposition.
  The Clifford homomorphisms $p^{\rho}_{\lambda}(X) \colon W_{\rho} \to W_{\lambda}$ are defined by the orthogonal projection $\Pi^{\rho}_{\lambda} \colon W_{\rho} \otimes \C^n \to W_{\lambda}$ as
  \[p^{\rho}_{\lambda}(X) \varphi := \Pi^{\rho}_{\lambda}(\varphi \otimes X), \quad \text{for } \varphi \in W_{\rho}, X \in \C^n.\]
  Then, the Clifford homomorphisms satisfy the following relations:
  \begin{enumerate}
    \item $\sum_i p^{\rho}_{\lambda}(e_i)^{\ast} p^{\rho}_{\lambda}(e_i) = \frac{\dim W_{\lambda}}{\dim W_{\rho}} \mathrm{id}_{W_{\rho}}$,
    \item $\sum_i p^{\rho}_{\lambda}(e_i) p^{\rho}_{\lambda}(e_i)^{\ast} = \mathrm{id}_{W_{\lambda}}$
    \item $p^{\rho}_{\lambda}(X)^{\ast} = -\sqrt{\frac{\dim W_{\lambda}}{\dim W_{\rho}}} p^{\lambda}_{\rho}(X)$
  \end{enumerate}
  where $\{e_i\}_i$ is an orthonormal basis of $\R^n$, and $p^{\rho}_{\lambda}(e_i)^{\ast}$ is the adjoint operator of $p^{\rho}_{\lambda}(e_i)$ with respect to the inner products of $W_{\rho}$ and $W_{\lambda}$.
\end{prop}
\begin{proof}
  For the first two equations see \cite{HT} or \cite{HomBW}.
  For the third equation, we note that there exists a non-zero constant $a \in \C$ such that
  \[p^{\rho}_{\lambda}(X)^{\ast} = a p^{\lambda}_{\rho}(X) \quad \forall X \in TM.\]
  This is because the two maps 
  \begin{align*}
    W_{\lambda} \to W_{\rho} \otimes \C^n, &\quad \varphi \mapsto \sum_i p^{\rho}_{\lambda}(e_i)^{\ast}\varphi \otimes e_i,\\
    W_{\lambda} \to W_{\rho} \otimes \C^n, &\quad \varphi \mapsto \sum_i p^{\lambda}_{\rho}(e_i) \varphi \otimes e_i    
  \end{align*}
  are $\Spin(n)$-equivariant and $W_{\rho} \otimes \C^n$ is multiplicity-free.
  Therefore, by the first and second equations of this proposition, we have $|a|^2 = \frac{\dim W_{\lambda}}{\dim W_{\rho}}$.
  Finally, by choosing the embedding $W_{\lambda} \to W_{\rho} \otimes \C^n$ appropriately (see Remark \ref{rema:choice_of_inclusion}), we can take $a = -\sqrt{\frac{\dim W_{\lambda}}{\dim W_{\rho}}}$.
\end{proof}

\begin{rema}
  \label{rema:choice_of_inclusion}
  The inner product preserving embedding $W_{\lambda} \to W_{\rho} \otimes \C^n$ has ambiguity up to a complex number with norm 1. So, the definition of generalized gradients and Clifford homomorphisms also has the same ambiguity. We choose the embedding so that the third equation in Proposition \ref{prop:relative_dimension_formula} holds. (The other two equations do not depend on the choice of the embedding.)
\end{rema}

\section{Higher spin Killing spinors}
\label{sec:Higher_Killing_spinors}
In this section, we consider an analogue of Killing spinors on higher spin spinor bundle $S_j$.
\begin{defi}
  A non-trivial section $\varphi \in \Gamma(S_j)$ is called a \emph{spin $(j + \frac{1}{2})$ Killing spinor} or \emph{higher spin  Killing spinor} (or simply \emph{Killing spinor}) if it satisfies the equation
\[\nabla_X \varphi = \mu p_j(X) \varphi \quad (\forall X \in TM),\]
where $\mu \in \C$ is a constant. The constant $\mu$ is called the \emph{Killing number} of $\varphi$. In particular, $\varphi$ is called a \emph{real Killing spinor} if $\mu$ is a non-zero real number, an \emph{imaginary Killing spinor} if $\mu$ is a purely imaginary number, and a \emph{parallel spinor} if $\mu = 0$.
\end{defi}

\begin{rema}
  \leavevmode
  \begin{enumerate}
    \item For $j=0$, since the Clifford homomorphism $p_0(X)$ is just the Clifford multiplication $p_0(X)\varphi = \tfrac{1}{\sqrt{n}}X\cdot\varphi$, the spin $\frac{1}{2}$ Killing spinors are the usual Killing spinors.
    \item If $j=0$ or $n=3$, the Killing number $\mu$ is restricted to be either real or purely imaginary. For $n=3$, we prove this fact in Corollary \ref{corr:Killingnum_on_3dim}.
  \end{enumerate}
\end{rema}

It is well-known that usual Killing spinors are eigenspinors of the Dirac operator. 
Similarly, we have the following proposition for the higher spin Killing spinors.

\begin{prop}
  \label{prop:higher_killing_spinor_eigenvalue}
  Let $\varphi$ be a spin $(j + \frac{1}{2})$ Killing spinor with Killing number $\mu$. 
  Then, $\varphi$ is an eigenspinor of the higher spin Dirac operator $D_j$ with eigenvalue $-\mu$.
\end{prop}
\begin{proof}
  By the definition of $D_j$ and the Killing spinor equation, we have
  \[
  D_j \varphi = \sum_i p_j(e_i) \nabla_{e_i} \varphi = \mu \sum_i p_j(e_i) p_j(e_i) \varphi.
  \]
  By Proposition \ref{prop:relative_dimension_formula}, we know $\sum_i p_j(e_i)p_j(e_i) = -1$. Thus we obtain $D_j \varphi = -\mu \varphi$.
\end{proof}

Since a higher spin Killing spinor $\varphi \in \Gamma(S_j)$ is parallel with respect to the modified connection $\widetilde{\nabla}_X = \nabla_X - \mu p_j(X)$, we immediately obtain the following proposition.
\begin{prop}
  \label{prop:dim_of_Killingspinors}
  The dimension of the space of spin $(j + \frac{1}{2})$ Killing spinors with Killing number $\mu$ is less than or equal to $\rank(S_j)$. And if $\mu$ is real, then $\langle \varphi,\psi \rangle$ is constant for such spinors $\varphi$ and $\psi$.
\end{prop}

It is well-known that for two Killing spinors $\varphi, \psi$ with the same real Killing number, a vector field $X_{\varphi, \psi}$ defined by
\[g(X_{\varphi, \psi}, Y) := \Re\langle Y \cdot \varphi,\psi \rangle \quad (\forall Y \in TM)\]
is a Killing vector field. Here, for $\alpha \in \C$, $\Re\alpha$ means the real part of $\alpha$.
We generalize this construction to the higher spin cases.

\begin{prop}
  \label{prop:higher_Killing_spinor_and_Killingtensor}
  For two higher spin Killing spinors $\varphi, \psi$ with the same real Killing number, a symmetric $(0, m)$-tensor field $K^m_{\varphi, \psi}$ defined by
  \[K^m_{\varphi, \psi}(X_1, \ldots, X_m) := \Re \langle p_j(X_1) \odot\cdots\odot p_j(X_m)\varphi, \psi \rangle\]
  is a Killing tensor field, namely $dK^m_{\varphi, \psi} = 0$ \cite{HMS}. Here, $\odot$ means the symmetrization of the operators, i.e.,
  \[p_j(X_1) \odot\cdots\odot p_j(X_m)\varphi := \frac{1}{m!} \sum_{\sigma \in \mathfrak{S}_m} p_j(X_{\sigma(1)}) \cdots p_j(X_{\sigma(m)})\varphi.\]
\end{prop}
\begin{proof}
  For any vector field $Y$, we have
  \begin{align*}
    (\nabla_Y K^m_{\varphi, \psi})(X_1, \ldots, X_m) &= \Re \langle p_j(X_1) \odot\cdots\odot p_j(X_m)\nabla_Y\varphi, \psi \rangle + \Re \langle p_j(X_1) \odot\cdots\odot p_j(X_m)\varphi,  \nabla_Y \psi \rangle \\
    &= \mu \Re \langle p_j(X_1) \odot\cdots\odot p_j(X_m) p_j(Y) \varphi, \psi \rangle + \mu \Re \langle (p_j(X_1) \odot\cdots\odot p_j(X_m) )\varphi, p_j(Y) \psi \rangle\\
    &= \mu \Re \langle p_j(X_1) \odot\cdots\odot p_j(X_m) p_j(Y) \varphi, \psi \rangle - \mu \Re \langle p_j(Y) (p_j(X_1) \odot\cdots\odot p_j(X_m) )\varphi, \psi \rangle.
  \end{align*}
  Hence, for any vector fields $X_0, X_1, \ldots, X_m$,
  \begin{align*}
    dK^m_{\varphi, \psi}(X_0, X_1, \ldots, X_m) &= \sum_{\sigma \in \mathfrak{S}_{m+1}} (\nabla_{X_{\sigma(0)}} K^m_{\varphi, \psi})(X_{\sigma(1)}, \ldots, X_{\sigma(m)}) \\
    &= \mu \sum_{\sigma \in \mathfrak{S}_{m+1}} \Re \langle (p_j(X_{\sigma(1)}) \odot\cdots\odot p_j(X_{\sigma(m)}))p_j(X_{\sigma(0)})\varphi, \psi \rangle\\
    &\hspace{50pt} - \mu \sum_{\sigma \in \mathfrak{S}_{m+1}}\Re \langle p_j(X_{\sigma(0)})(p_j(X_{\sigma(1)}) \odot\cdots\odot p_j(X_{\sigma(m)}))\varphi,\psi \rangle \\
    &= 0.
  \end{align*}
\end{proof}
\begin{rema}
  In the spin $\tfrac{1}{2}$ case, we have
  \[K^{2m}_{\varphi, \psi} = 2^m \underbrace{g \odot \cdots \odot g}_m \odot K^0_{\varphi, \psi}, \quad K^{2m+1}_{\varphi, \psi} = 2^m \underbrace{g \odot \cdots \odot g}_{m} \odot K^1_{\varphi, \psi},\]
  where $K^0_{\varphi, \psi}$ is just a real part of the inner product of $\varphi$ and $\psi$,  and $K^1_{\varphi, \psi}$ is a dual 1-form of the Killing vector field $X_{\varphi, \psi}$ induced by $\varphi, \psi$. Hence this construction only becomes valuable in the higher spin cases.
  We discuss a relation between higher spin Killing spinors and Killing tensors in more detail in \S \ref{sec:Killing_tensor_on_S3}.
\end{rema}

Now we should argue the existence of higher spin Killing spinors.

\begin{prop}
  \label{prop:space of the higher Killing spinor}
  Let $\varphi$ be a spin $(j + \frac{1}{2})$ Killing spinor.
  Then $\varphi$ is in $\ker T^+_j \cap \ker T^-_j \cap \ker U_j$.
  In particular, $\varphi$ is in $\ker T^+_j \cap \ker T^-_j$ if $n = 3$.
\end{prop}
\begin{proof}
  Here we only prove $\varphi \in \ker T^+_j$. The other cases can be proved in the same way.
  By the definition of $T^+_j$ and the Killing spinor equation, we have
  \[T^+_j \varphi = \sum_i p^+_j(e_i) \nabla_{e_i} \varphi = \mu \sum_i p^+_j(e_i) p_j(e_i) \varphi.\]
  The map $\sum_i p^+_j(e_i) p_j(e_i)$ is $\Spin(n)$-equivariant from $W_j$ to $W_{j+1}$. Since $W_j$ and $W_{j+1}$ are irreducible and non-isomorphic, Schur's lemma implies $\sum_i p^+_j(e_i) p_j(e_i) = 0$. Thus we obtain $T^+_j \varphi = 0$.
\end{proof}

In this paper, we have obtained examples of Killing spinors with spin $\ge 3/2$ and $\mu \neq 0$ only on 3-dimensional manifolds (see \S \ref{sec:higher_killing_on_S3}, \ref{sec:higher_killing_on_H3}). 
This is because the existence of such spinors on manifolds of dimension 4 or higher imposes extremely strong geometric constraints. 
It is currently unknown whether such manifolds actually exist. 
To illustrate this situation, we present the following proposition.

\begin{prop}
  Let $(M^n, g)$ be an $n$-dimensional Riemannian spin manifold of constant sectional curvature $K = c$. If $M$ admits a non-trivial spin $(j + \frac{1}{2})$ Killing spinor with Killing number $\mu$, then at least one of the following holds: (1) $c = 0$ and $\mu = 0$, (2) $n=3$, (3) $j=0$.
  In particular, the sphere $\sphere{n}$ and the hyperbolic space $\hb{n}$ do not admit any non-trivial spin $(j + \frac{1}{2})$ Killing spinors if $n \geq 4$ and $j \geq 1$.
\end{prop}
\begin{proof}
  According to \cite{HT}, when $(M^n, g)$ is of constant sectional curvature $K = c$, the following identity holds on $\Gamma(S_j)$:
  \begin{align*}
    \nabla^{\ast}\nabla + &\left(j(n+j-1) + \frac{n(n-1)}{8}\right)c  \\
    &= \frac{(n+2j)(n-2)}{n+2j-2}D_j^2 + \frac{4(n+j-2)}{n+2j-2}(T^-_j)^{\ast} T^-_j + \left(j(n+j-2) - \frac{n(n-1)}{8}\right)c \\
    &= \frac{4(j+1)}{n+2j}(T_j^+)^{\ast} T_j^+ + \frac{(n+2j-2)(n-2)}{n+2j}D_j^2 + \left((j+1)(n+j-1) - \frac{n(n-1)}{8}\right)c 
  \end{align*}
  Here, $\nabla^{\ast}\nabla = (T^-_j)^{\ast} T^-_j + D_j^2 + (T^+_j)^{\ast} T^+_j + U_j^{\ast}U_j$ is the connection Laplacian on $S_j$.
  Let $\varphi$ be a non-trivial spin $(j + \frac{1}{2})$ Killing spinor with Killing number $\mu$.
  By Proposition \ref{prop:higher_killing_spinor_eigenvalue} and Proposition \ref{prop:space of the higher Killing spinor}, substituting $\varphi$ into the above equations, we obtain
  \begin{align}
    \mu^2 \varphi  + &\left(j(n+j-1) + \frac{n(n-1)}{8}\right)c\varphi \label{eq:sec3prop-1}\\
    &= \mu^2\frac{(n+2j)(n-2)}{n+2j-2}\varphi + \left(j(n+j-2) - \frac{n(n-1)}{8}\right)c \varphi, \label{eq:sec3prop-2}\\
    &= \mu^2\frac{(n+2j-2)(n-2)}{n+2j}\varphi + \left((j+1)(n+j-1) - \frac{n(n-1)}{8}\right)c\varphi.\label{eq:sec3prop-3}
  \end{align}
  From two equations \eqref{eq:sec3prop-1} = \eqref{eq:sec3prop-2} and \eqref{eq:sec3prop-1} = \eqref{eq:sec3prop-3}, we have
  \begin{gather}
      \mu^2\left(1-\frac{(n+2j)(n-2)}{n+2j-2}\right) + \left(j + \frac{n(n-1)}{4}\right)c = 0, \label{eq:sec3prop-4}\\
  \mu^2\left(1-\frac{(n+2j-2)(n-2)}{n+2j}\right) + \left(-(n+j-1) + \frac{n(n-1)}{4}\right)c = 0. \label{eq:sec3prop-5}
  \end{gather}
  Combining \eqref{eq:sec3prop-4} and \eqref{eq:sec3prop-5}, we obtain
  \begin{equation}
    c = 4\mu^2\frac{n-2}{(n+2j-2)(n+2j)}. \label{eq:sec3prop-6}
  \end{equation}
  By substituting \eqref{eq:sec3prop-6} into \eqref{eq:sec3prop-4} and simplifying, we get
  \[4\mu^2 j(n-3)(j+n-1)=0.\]
  Therefore, at least one of the following holds: (1) $\mu = 0$, (2) $n=3$, (3) $j=0$.
  If $\mu = 0$, then we have $c = 0$ by \eqref{eq:sec3prop-6}.
\end{proof}

\section{Spin geometry with higher spin on 3-dimensional manifolds}
\label{sec:3dimcase}
\subsection{Weitzenb\"ock formulas}
\label{sec:ingredients_of_higherspin_on_3dim}
In this subsection we review Weitzenb\"ock formulas on 3-dimensional manifolds discussed in \cite{Hom3dim}.
First, we should investigate representations of $\Spin(3) \cong \SU(2)$ more precisely.
In the 3-dimensional case, the spin $(j + \frac{1}{2})$ representation ($\pi_j$, $W_j$) is the $(2j+2)$-dimensional irreducible representation of $\SU(2)$.
In other words, $W_j$ is an irreducible representation of $\SU(2)$ with highest weight $2j + 1$.
We write the infinitesimal representation of $\pi_j$ by the same symbol $\pi_j \colon \su(2) \to \End(W_j)$.
As mentioned in \S \ref{sec:preliminary}, $W_j \otimes \C^3$ can be decomposed into the three irreducible components:
\[W_j \otimes \C^3 \cong W_{j+1} \oplus W_j \oplus W_{j-1}.\]
The Clifford homomorphisms have been defined by
\begin{align*}
  p_j(X) \varphi &= \Pi_j(\varphi \otimes X), \quad p_j(X)\colon W_j \to W_j,\\
  p^+_{j}(X) \varphi &= \Pi^+_j(\varphi \otimes X), \quad p^+_{j}(X)\colon W_j \to W_{j+1},\\
  p^-_{j}(X) \varphi &= \Pi^-_j(\varphi \otimes X), \quad p^-_{j}(X)\colon W_j \to W_{j-1}.
\end{align*}
In the 3-dimensional case, the Clifford homomorphism $p_j$ is consistent with the infinitesimal representation $\pi_j$ up to a constant.
We set the $\SU(2)$-equivariant isometric inclusion $\iota \colon W_j \to W_j \otimes \C^3$ by
\begin{equation}
  \iota(\varphi) = -\frac{1}{\sqrt{(2j+1)(2j+3)}} \sum_i \pi_j(\sigma_i)\varphi \otimes e_i, \label{eq:embedding_of_W_j}
\end{equation}
where $\sigma_1, \sigma_2, \sigma_3$ are the Pauli matrices (see Remark \ref{rema:identification_su2_and_R3}) and the coefficient $((2j+1)(2j+3))^{-\tfrac{1}{2}}$ comes from the action of Casimir element $c = \sigma_1\sigma_1 + \sigma_2\sigma_2 + \sigma_3\sigma_3$:
  \[\pi_j(c) = \sum_i \pi_j(\sigma_i)\pi_j(\sigma_i) = -(2j+1)(2j+3).\]
  One can easily check that $\iota$ preserves the inner product by using this fact (see also Remark \ref{rema:choice_of_inclusion}).
\begin{prop}
  For all $X \in \R^3 \cong \su(2)$ and $\varphi \in W_j$, we have
  \[p_j(X) \varphi = \frac{1}{\sqrt{(2j+1)(2j+3)}} \pi_j(X) \varphi.\]
\end{prop}
\begin{proof}
  For all $\psi \in W_j$, we have 
  \begin{align*}
    \langle p_j(X)\varphi, \psi \rangle = \langle \varphi \otimes X, \iota(\psi) \rangle &= -\frac{1}{\sqrt{(2j+1)(2j+3)}} \sum_i \langle \varphi, \pi_j(\sigma_i) \psi \rangle \langle X, e_i \rangle \\
    &= -\frac{1}{\sqrt{(2j+1)(2j+3)}} \langle \varphi, \pi_j(X) \psi \rangle\\
    &= \frac{1}{\sqrt{(2j+1)(2j+3)}} \langle \pi_j(X)\varphi, \psi \rangle.
  \end{align*}
\end{proof}

\begin{rema}
  \label{rema:identification_su2_and_R3}
  The identification $\R^3 \cong \su(2)$ is given by Pauli matrices $\sigma_1, \sigma_2, \sigma_3$:
  \[
  e_1 \mapsto \sigma_1 = \begin{pmatrix}
    \sqrt{-1} & 0 \\
    0 & -\sqrt{-1}
  \end{pmatrix}, \quad
  e_2 \mapsto \sigma_2 = \begin{pmatrix}
    0 & 1 \\
    -1 & 0
  \end{pmatrix}, \quad
  e_3 \mapsto \sigma_3 = \begin{pmatrix}
    0 & \sqrt{-1} \\
    \sqrt{-1} & 0
  \end{pmatrix},
  \]
  where $\{e_1, e_2, e_3\}$ is the standard basis of $\R^3$.
\end{rema}

Since $\pi_j$ differs from $p_j$ only by a constant factor and is more convenient for calculations, we will use $\pi_j$ as the Clifford homomorphism in what follows.
For simplicity, we normalize the other Clifford homomorphisms as follows:
  \[\pi^+_{j}(X) := \sqrt{\frac{4(j+1)}{2j+3}} p^+_{j}(X), \quad \pi^-_{j}(X) := \sqrt{\frac{4(j+1)}{2j+1}} p^-_{j}(X).\]
We also normalize the higher spin Dirac operator and two twistor operators as follows:
\begin{align*}
  D_j &:= \frac{1}{2j+1} \sum_i \pi_j(e_i)\nabla_{e_i} \colon \Gamma(S_j) \to \Gamma(S_j), \\
  T^{\pm}_j &:= \sum_i\pi_j^{\pm}(e_i)\nabla_{e_i} \colon \Gamma(S_j) \to \Gamma(S_{j \pm 1}).
\end{align*}

\begin{prop}
  Under our normalization, we have
  \[\pi_j(X)^{\ast} = -\pi_j(X), \quad \pi^+_j(X)^* = -\pi^-_{j+1}(X), \quad \pi^-_j(X)^* = -\pi^+_{j-1}(X).\]
  Thus the differential operators satisfy
  \[(D_j)^* = D_j,\quad (T^+_j)^* = T^-_{j+1}, \quad (T^-_j)^* = T^+_{j-1},\]
  where $P^\ast$ is the formal adjoint operator for a differential operator $P$.
\end{prop}
\begin{proof}
  The first equation is clear since $\pi_j$ is a unitary representation of $\su(2)$.
  We only prove the second equation.
  By Proposition \ref{prop:relative_dimension_formula}, we have
  \[\pi_j^+(X)^{\ast} = \sqrt{\frac{4(j+1)}{2j+3}} p^+_{j}(X)^{\ast} = -\sqrt{\frac{4(j+1)}{2j+3}}\sqrt{\frac{2j+4}{2j+2}}p^-_{j+1}(X) = -\sqrt{\frac{4(j+2)}{2j+3}}p^-_{j+1}(X) = -\pi^-_{j+1}(X).\]
  The formal adjoint operator of $T^+_j$ can be computed locally
  \[(T^+_j)^{\ast} = -\sum_i \pi_j^+(e_i)^{\ast} \nabla_{e_i} = \sum_i \pi^-_{j+1}(e_i) \nabla_{e_i} = T^-_{j+1}.\]
  The other equations can be proved in the same way.
\end{proof}

The tangent bundle $TM$ admits a Lie bracket, since $TM \cong \Spin(M) \times_{\Ad} \su(2)$. The Lie bracket $[X,Y]_{\g}$ is induced from the Lie algebra $\su(2)$ fiberwise. We take positively oriented local orthonormal frame $\{e_1, e_2, e_3\}$ of $TM$. Then these satisfy the following relations:
\[[e_1, e_2]_{\g} = 2e_3, \quad [e_2, e_3]_{\g} = 2e_1, \quad [e_3, e_1]_{\g} = 2e_2.\]
This relation can be expressed more simply by using the Hodge star operator $\ast$ as follows:
\begin{equation}
  \label{eq:g-bracket}
  \frac{1}{2}[X, Y]_{\g} = \ast (X \wedge Y).
\end{equation}
Here, we identify $T^{\ast}M \cong TM$ by the metric $g$.

A basic tool in spin geometry is Clifford algebra, which is an algebra generated by the principal symbol of the Dirac operator. Similarly it is important to study how $\pi_j, \pi_j^+$ and $\pi_j^-$ are related to each other. From universal Weitzenb\"ock formula \cite{HomBW} or a direct calculation in \cite{Hom3dim}, we know the following lemma.
\begin{lemm}
  \label{lemma:algebraic_weitzenbock_formula}
  For any $j \geq 0$, the Clifford homomorphisms $\pi_j, \pi_j^+, \pi_j^-$ satisfy the following two equations:
  \begin{enumerate}
    \item 
      \[\frac{2j+3}{4(j+1)}\pi_{j+1}^-(e_k)\pi_j^+(e_l) + \frac{1}{(2j+3)(2j+1)}\pi_j(e_k)\pi_j(e_l)+\frac{2j+1}{4(j+1)}\pi_{j-1}^+(e_k)\pi_{j}^-(e_l) = - \delta_{kl}\]
    \item
      \[\hspace{-30pt}-\left(j+\frac{1}{2}\right)\frac{2j+3}{4(j+1)}\pi_{j+1}^-(e_k)\pi_j^+(e_l) + \frac{1}{(2j+3)(2j+1)}\pi_j(e_k)\pi_j(e_l) + \left(j+\frac{3}{2}\right)\frac{2j+1}{4(j+1)}\pi_{j-1}^+(e_k)\pi_{j}^-(e_l) = \frac{1}{4}\pi_j([e_k, e_l]_{\g})\]
  \end{enumerate}
\end{lemm}
By Lemma \ref{lemma:algebraic_weitzenbock_formula}, we also have
\begin{equation}
  \frac{1}{2(2j+1)}\pi_j(e_k)\pi_j(e_l) + \frac{2j+1}{2}\pi^+_{j-1}(e_k)\pi^-_j(e_l) = \frac{1}{4}\pi_j([e_k, e_l]_{\g}) - \left(j+\frac{1}{2}\right)\delta_{kl}.  \label{eq:comm_rel_of_pipm}
\end{equation}
Lemma \ref{lemma:algebraic_weitzenbock_formula} gives relations between the generalized gradients.
\begin{prop}
  \label{prop:Bochner-Weitzenbock}
  For any $j \geq 0$, the generalized gradients on $S_j$ satisfy the following relations:
  \begin{enumerate}
    \item \[\frac{2j+3}{4(j+1)}T^-_{j+1} T^+_j + \frac{2j+1}{2j+3}D_j^2 + \frac{2j+1}{4(j+1)}T^+_{j-1} T^-_j = \nabla^{\ast} \nabla,\]
  \item \begin{equation}
     -\left(j+\frac{1}{2}\right)\frac{2j+3}{4(j+1)}T^-_{j+1} T^+_j + \frac{2j+1}{2j+3}D_j^2 + \left(j+\frac{3}{2}\right)\frac{2j+1}{4(j+1)}T^+_{j-1} T^-_j = q(R). \label{eq:wetzenbock_formula}
  \end{equation}
  \end{enumerate}
  A curvature action $q(R) \in \Gamma(\End(S_j))$ is a symmetric endomorphism on $S_j$ defined by
  \[q(R) = \frac{1}{8}\sum_{k,l} \pi_j([e_k, e_l]_{\g})R_j(e_k, e_l),\]
  where $R$ is the Riemannian curvature tensor and $R_j$ is the curvature tensor of the bundle $S_j$.
\end{prop}
\begin{proof}
  We prove the first equation.
  Fix a point $x \in M$ and take a local orthonormal frame $\{e_i\}$ such that $\nabla e_i = 0$ at $x$.
  Then, at the point $x$, we have
    \begin{gather*}
      T^-_{j+1} T^+_j = \sum_{k,l} \pi^-_{j+1}(e_k) \pi^+_j(e_l) \nabla_{e_k} \nabla_{e_l},\\
      D_j^2 = \frac{1}{(2j+1)^2} \sum_{k,l} \pi_j(e_k) \pi_j(e_l) \nabla_{e_k} \nabla_{e_l},\\
      T^+_{j-1} T^-_j = \sum_{k,l} \pi^+_{j-1}(e_k) \pi^-_j(e_l) \nabla_{e_k} \nabla_{e_l}.
    \end{gather*}
  By using the first equation in Lemma \ref{lemma:algebraic_weitzenbock_formula}, we have
  \begin{align*}
    &\frac{2j+3}{4(j+1)}T^-_{j+1} T^+_j + \frac{2j+1}{2j+3}D_j^2 + \frac{2j+1}{4(j+1)}T^+_{j-1} T^-_j \\
    &= \sum_{k,l} \left(\frac{2j+3}{4(j+1)}\pi^-_{j+1}(e_k) \pi^+_j(e_l) + \frac{1}{(2j+3)(2j+1)}\pi_j(e_k) \pi_j(e_l) + \frac{2j+1}{4(j+1)}\pi^+_{j-1}(e_k) \pi^-_j(e_l)\right) \nabla_{e_k} \nabla_{e_l} \\
    &= \sum_{k} -\nabla_{e_k} \nabla_{e_k} = \nabla^{\ast} \nabla.
  \end{align*}
  The second equation can be proved in the same way from the second equation in Lemma \ref{lemma:algebraic_weitzenbock_formula}.
  For the right hand side, we note that
  \[\frac{1}{4} \sum_{k,l} \pi_j([e_k, e_l]_{\g}) \nabla_{e_k} \nabla_{e_l} = \frac{1}{8} \sum_{k,l} \pi_j([e_k, e_l]_{\g}) \left(\nabla_{e_k} \nabla_{e_l} - \nabla_{e_l} \nabla_{e_k}\right) = \frac{1}{8} \sum_{k,l} \pi_j([e_k, e_l]_{\g})R_j(e_k, e_l).\]
  Symmetry of $q(R)$ follows from the fact that $\pi_j([e_k, e_l]_{\g})$ and $R_j(e_k, e_l)$ are skew-adjoint.
\end{proof}

Let us consider the curvature tensor $R_j$ on $S_j$.
With respect to a local orthonormal frame $\{e_1, e_2, e_3\}$, $R_j$ can be locally expressed as follows:
\begin{align*}
  R_j(X,Y) &= \frac{1}{2}\sum_{k,l} g( R(X,Y)e_l,e_k)\pi_j(e_k \wedge e_l)\\
  &= \frac{1}{2}\sum_{k,l} g( R(X,Y)e_l,e_k)\pi_j([e_k, e_l]_{\g})\\
  &= \frac{1}{4}\sum_{k,l} g(R(X,Y)e_l, e_k)\pi_j(\sigma_l \sigma_k)\\
  &= \frac{1}{4}\sum_{\tau \in \mathfrak{S}_3} \sgn(\tau) g(R(X,Y)e_{\tau(1)}, e_{\tau(2)})\pi_j(e_{\tau(3)}),
\end{align*}
where we use the identification $\su(2)\cong \so(3)$ given by $[e_i,e_j]_{\g} \mapsto e_i\wedge e_j$.
In particular, we have
\begin{equation}
  R_j(e_1, e_2) = -\frac{1}{4}\scal \pi_j(e_3) + \frac{1}{2} \pi_j(\Ric(e_3)), \label{eq:loc_expression_of_curvature_on_Sj}
\end{equation}
and cyclic permutations for $R_j(e_2, e_3), R_j(e_3, e_1)$.
Then we obtain another expression of the curvature action $q(R)$ in $\Gamma(\End(S_j))$.
\begin{prop}
  \label{prop:expression_of_q(R)}
  The curvature action $q(R)$ can be expressed as
  \begin{align*}
    q(R) &= \frac{\scal}{8}(2j+3)(2j+1) + \frac{1}{4}\sum_{i} \pi_j(e_i)\pi_j(\Ric(e_i)).
  \end{align*}
\end{prop}

Next we shall study twisted Weitzenb\"ock formulas.
From the general theory of Clifford homomorphisms \cite{HomTD} or a direct calculation in \cite{Hom3dim}, we know the following lemma.
\begin{lemm}
  \label{lemm:alg_twisted_Weitzenbock}
  The following two identities for linear maps $W_j$ to $W_{j \pm 1}$ hold:
  \begin{enumerate}
    \item \[\pi_{j+1}(e_k)\pi_j^+(e_l) - \pi_j^+(e_k)\pi_j(e_l) = \frac{2j+3}{2}\pi_j^+([e_k, e_l]_{\g}),\]
    \item \[\pi_{j}^-(e_k)\pi_{j}(e_l) - \pi_{j-1}(e_k)\pi_{j}^-(e_l) = \frac{2j+1}{2}\pi_{j}^-([e_k, e_l]_{\g}).\]
  \end{enumerate}
\end{lemm}
Note that the second equation is obtained by taking the adjoint of the first equation.
From these relations, we have the following proposition by the same way to the proof of Proposition \ref{prop:Bochner-Weitzenbock}.

\begin{prop}
  We define two curvature actions $q^+(R), q^-(R) \in \Gamma(\Hom(S_j, S_{j \pm 1}))$ by
  \[q^{\pm}(R) = \frac{1}{4}\sum_{k,l} \pi_{j}^{\pm}([e_k, e_l]_{\g})R_j(e_k, e_l).\]
  Then, following two equations hold:
 \begin{enumerate}
  \item
  \begin{equation}  
    q^+(R) = D_{j+1}T_j^+-\dfrac{2j+1}{2j+3}T_j^+D_j, \label{eq:twisted_weitzenbock_formula+}
  \end{equation}
  \item 
  \begin{equation}
    q^-(R) = T_j^-D_j-\dfrac{2j-1}{2j+1}D_{j-1}T_j^-. \label{eq:twisted_weitzenbock_formula-}
  \end{equation}
 \end{enumerate}
\end{prop}
\begin{proof}
  We only prove the first equation. The second equation can be proved in the same way.
  Fix a point $x \in M$ and take a local orthonormal frame $\{e_i\}$ such that $\nabla e_i = 0$ at $x$.
  Then, at the point $x$, we have
    \begin{gather*}
      D_{j+1}T_j^+ = \frac{1}{2j+3} \sum_{k,l} \pi_{j+1}(e_k) \pi^+_j(e_l) \nabla_{e_k} \nabla_{e_l},\\
      T_j^+D_j = \frac{1}{2j+1} \sum_{k,l} \pi^+_j(e_k) \pi_j(e_l) \nabla_{e_k} \nabla_{e_l}.
    \end{gather*}
  By using the first equation in Lemma \ref{lemm:alg_twisted_Weitzenbock}, we have
  \begin{align*}
    D_{j+1}T_j^+ - \dfrac{2j+1}{2j+3}T_j^+D_j &= \sum_{k,l} \left(\frac{1}{2j+3}\pi_{j+1}(e_k) \pi^+_j(e_l) - \frac{1}{2j+3}\pi^+_j(e_k) \pi_j(e_l)\right) \nabla_{e_k} \nabla_{e_l}\\
    &= \frac{1}{4}\sum_{k,l} \pi_j^+([e_k, e_l]_{\g}) (\nabla_{e_k} \nabla_{e_l} - \nabla_{e_l} \nabla_{e_k}) = q^+(R).
  \end{align*}
\end{proof}

\subsection{The basics on higher spin Killing spinors on 3-dimensional manifolds}
\label{sec:higher_spin_killing_on_3mfd}
Now we normalize the Killing number $\mu$ such that
\[\nabla_X \varphi = \mu \pi_j(X) \varphi \quad (\forall X \in TM),\]
because of the normalization of Clifford homomorphisms.

\begin{exam}
  A trivial example of higher spin Killing spinors is a higher parallel spinor (i.e. $\mu = 0$) on the 3-dimensional flat torus $T^3$. We can choose a spin structure on $T^3$ as the trivial principal $\SU(2)$-bundle over $T^3$.
  The spin $(j + \frac{1}{2})$ spinor bundle $S_j$ on $T^3$ is trivial and therefore, all constant spinors are higher spin parallel spinors.
  We give non-trivial examples of higher spin Killing spinors with non-zero Killing number on the 3-sphere $\sphere{3}$ in \S \ref{sec:higher_killing_on_S3}, and on the hyperbolic 3-space $\hb{3}$ in \S \ref{sec:higher_killing_on_H3}.
\end{exam}

We restate some properties of higher spin Killing spinors from \S\ref{sec:Higher_Killing_spinors} in the 3-dimensional case.
By Proposition \ref{prop:dim_of_Killingspinors}, the dimension of the space of spin $(j + \frac{1}{2})$ Killing spinors with Killing number $\mu$ is less than or equal to $2j + 2$.
The next proposition follows from Proposition \ref{prop:higher_killing_spinor_eigenvalue} and Proposition \ref{prop:space of the higher Killing spinor}.
\begin{prop}
  The spin $(j + \frac{1}{2})$ Killing spinor $\varphi$ with Killing number $\mu$ satisfies $D_j \varphi = -(2j+3)\mu\varphi$ and $\varphi$ is in $\ker T^+_j \cap \ker T^-_j$.
\end{prop}

From now on, we study a manifold admitting a spin $(j + \frac{1}{2})$ Killing spinor.
In the usual spinor case ($j = 0$), it is well-known that the existence of a Killing spinor implies that the manifold is an Einstein manifold. We shall show the same result for higher spin Killing spinors.
First, we derive an integrability condition for higher spin Killing spinors.
\begin{prop}
  Let $M$ be a 3-dimensional spin manifold admitting a spin $(j + \frac{1}{2})$ Killing spinor $\varphi$ with Killing number $\mu$. Then, $\varphi$ satisfies
  \[\pi_j\left(\Ric(X) - \frac{1}{2}(\scal - 8\mu^2)X\right)\varphi = 0 \quad (\forall X \in TM)\]
\end{prop}
\begin{proof}
  Fix any point $x \in M$ and take a local orthonormal frame $\{e_1, e_2, e_3\}$ around $x$ such that $(\nabla e_i)_x = 0$. Then we have
  \[R_j(e_1, e_2)\varphi = \mu^2(\pi_j(e_2)\pi_j(e_1) - \pi_j(e_1)\pi_j(e_2))\varphi = -\mu^2 \pi_j([e_1, e_2]_{\g})\varphi = -2\mu^2 \pi_j(e_3)\varphi.\]
  Hence, by equation \eqref{eq:loc_expression_of_curvature_on_Sj}, we obtain
    \[\pi_j\left(\Ric(e_3) - \frac{1}{2}(\scal - 8\mu^2)e_3\right)\varphi = 0.\]
  By a similar argument, the same equation holds for $e_1$ and $e_2$.
\end{proof}

From this proposition, we obtain the curvature action on the higher spin Killing spinors.
\begin{corr}
  \label{corr:Killingnum_on_3dim}
  For spin $(j+\frac{1}{2})$ Killing spinor $\varphi$, we have
  \[q(R)\varphi = \mu^2(2j+3)(2j+1)\varphi.\]
  In particular, $\mu$ must be a real number or purely imaginary number.
  We also have $q^{\pm}(R)\varphi = 0$.
\end{corr}
\begin{proof}
By Proposition \ref{prop:expression_of_q(R)}, we have
\begin{align*}
  0 &= \sum_i \pi_j(e_i)\pi_j\left(\Ric(e_i) - \frac{1}{2}(\scal - 8\mu^2)e_i\right)\varphi \\
  &= 4q(R)\varphi - \frac{1}{2}(2j+3)(2j+1)\scal\varphi + \frac{1}{2}(\scal - 8\mu^2)(2j+3)(2j+1)\varphi \\
  &= 4q(R)\varphi - 4\mu^2(2j+3)(2j+1)\varphi.
\end{align*}
Since $q(R)$ is a symmetric endomorphism, eigenvalues of $q(R)$ are real numbers. Thus, $\mu^2$ must be a real number.
$q^{\pm}(R)\varphi = 0$ follows from the fact that $\varphi \in \ker T^+_j \cap \ker T^-_j$ and \eqref{eq:twisted_weitzenbock_formula+}, \eqref{eq:twisted_weitzenbock_formula-}.
\end{proof}
Now we prove that a 3-dimensional manifold admitting a higher spin Killing spinor is an Einstein manifold.
\begin{theoremalpha}
\label{theo:3dim_Killing_and_Einstein}
  Let $(M,g)$ be a 3-dimensional spin manifold. If $M$ admits a spin $(j + \frac{1}{2})$ Killing spinor $\varphi$ with Killing number $\mu$, then $(M,g)$ is an Einstein manifold, and hence, $M$ is of constant curvature with $\scal = 24\mu^2$.
\end{theoremalpha}
\begin{proof}
  A Killing spinor $\varphi$ is parallel with respect to $\widetilde{\nabla}_X = \nabla_X - \mu\pi_j(X)$ so that $\varphi$ has no zeros. The above integrability condition implies that there is a non-trivial solution for each point $x \in M$. Then we have
  \[\det\left(\pi_j\left(\Ric(X) - \frac{1}{2}(\scal - 8\mu^2)X\right)\right) = 0 \quad (\forall X \in T_xM)\]
  for each point $x \in M$.
  First, we know there is a non-zero constant $c_j$ such that
  \[\det(\pi_j(e_1)) = c_j \neq 0\]
  because there is no weight vector with weight zero for the spin $(j + \frac{1}{2})$ representation.
  Next, for any non-zero vector $Y \in T_xM$, we can take $g \in \SU(2)$ such that $\|Y\|ge_1g^{-1} = Y$. Since the dimension of the spin $(j + \frac{1}{2})$ representation space $W_j$ is $2j + 2$, we have
  \[\det(\pi_j(Y)) = \|Y\|^{2j+2}c_j.\]
  Put $Y = \Ric(X) - \frac{1}{2}(\scal - 8\mu^2)X$ and we have
  \[0 = \det\left(\pi_j\left(\Ric(X) - \frac{1}{2}(\scal - 8\mu^2)X\right)\right) = c_j\left\|\Ric(X) - \frac{1}{2}(\scal - 8\mu^2)X\right\|^{2j+2}.\]
  Then we have
  \[\Ric(X) - \frac{1}{2}(\scal - 8\mu^2)X = 0\]
  for all $X \in TM$. This means that $(M,g)$ is an Einstein manifold. Moreover, taking the trace of both sides of this equation, the scalar curvature of $M$ satisfies
  \[\scal = \frac{3}{2}(\scal - 8\mu^2),\]
  that is, $\scal = 24\mu^2$.
\end{proof}

\subsection{Cone construction}
\label{sec:cone_construction}
In the spin $\tfrac{1}{2}$ case, C. B\"ar proved that there is a one-to-one correspondence between real Killing spinors on a Riemannian spin manifold and parallel spinors on the cone over the manifold for the classification of the manifolds admitting real Killing spinors \cite{BarKilling}.
We shall extend this correspondence to higher spin Killing spinors on 3-dimensional manifolds.

Let $(M, g)$ be a 3-dimensional Riemannian spin manifold and $(\cone{M}, \cone{g})$ the cone over $(M,g)$ defined by
\[\cone{M} = M \times \R_+, \quad \cone{g} = r^2 g + dr^2.\]
The cone $\cone{M}$ is a 4-dimensional Riemannian spin manifold. Indeed, if $\theta \colon P_{\Spin} M \to P_{\SO} M$ is a spin structure on $M$, then the spin structure on $\cone{M}$ is given by
\[P_{\Spin} \cone{M} := \pi^\ast(P_{\Spin} M \times_{\Spin(3)} \Spin(4)),\]
\[\cone{\theta} \colon P_{\Spin} \cone{M} \to P_{\SO} \cone{M}, \quad ((x, r), [\tilde{u}_x, g]) \mapsto \left(\frac{1}{r}\theta(\tilde{u}_x), \frac{\partial}{\partial r}\right)\xi(g),\]
\[(x, r) \in \cone{M},\,\, \tilde{u}_x \in (P_{\Spin} M)_x,\,\, g \in \Spin(4),\]
where $\pi \colon \cone{M} \to M$ is the natural projection and $\xi \colon \Spin(4) \to \SO(4)$ is the covering map.

Since $\Spin(4) \cong \SU(2) \times \SU(2)$, the irreducible representations of $\Spin(4)$ are given by the (outer) tensor products of the irreducible representations of $\SU(2)$. We note that the differential of this isomorphism yields the Lie algebra isomorphism $\so(4) \cong \so(3) \oplus \so(3)$, which is given by
\begin{align*}
  \text{first $\so(3)$} &= \operatorname{span}_{\C} \{ e_1 \wedge e_2 + e_3 \wedge e_4,\,\, e_2 \wedge e_3 + e_1 \wedge e_4,\,\, e_3 \wedge e_1 + e_2 \wedge e_4 \},\\
  \text{second $\so(3)$} &= \operatorname{span}_{\C}\{ e_1 \wedge e_2 - e_3 \wedge e_4,\,\, e_2 \wedge e_3 - e_1 \wedge e_4,\,\, e_3 \wedge e_1 - e_2 \wedge e_4 \}.
\end{align*}
Namely, the first $\so(3)$ corresponds to the self-dual 2-forms and the second $\so(3)$ corresponds to the anti-self-dual 2-forms.
We consider the vector bundle on $\cone{M}$ associated to the representation $(\pi_{j,0}, W_{j,0} = W_j \boxtimes \C)$, which is denoted by $\cone{S}_{j,0} = P_{\Spin}\cone{M} \times_{\pi_{j,0}} W_{j,0}$.
The vector bundle $\cone{S}_{j,0}$ is naturally isomorphic to the pullback bundle $\pi^{\ast} S_j$ because the restriction of $W_{j, 0}$ to $\{(g,g) \in \SU(2) \times \SU(2) \mid g \in \SU(2)\} \cong \SU(2)$ is equivalent to $W_j$.

Now we study a relation between the covariant derivatives on $S_j$ and $\cone{S}_{j,0}$.
We note that $\bigwedge^2 T\cone{M}$ acts on $\cone{S}_{j, 0}$ by $\pi_{j, 0} \colon \spin(4) \cong \so(4) \to \End(W_{j,0})$. For a positively oriented local orthonormal frame $(X_1, X_2, X_3)$ on $M$, we take a local orthonormal frame $(\cone{X_1}, \cone{X_2}, \cone{X_3}, \cone{X_4}) = (\frac{1}{r}X_1, \frac{1}{r}X_2, \frac{1}{r}X_3, \partial _r)$ on $\cone{M}$. 
Then, the covariant derivative $\cone{\nabla}$ on $\cone{S}_{j,0}$ is locally expressed as
\begin{align*}
  \cone{\nabla}_{X_1} &= X_1 + \frac{1}{2} \sum_{1 \leq k,l \leq 4} \cone{g}(\cone{\nabla}_{X_1} \cone{X_k}, \cone{X_l}) \pi_{j,0}(\cone{X_k} \wedge \cone{X_l})\\
  &= X_1 + \frac{1}{2r^2} \sum_{1 \leq k,l \leq 3} \cone{g}(\cone{\nabla}_{X_1} X_k, X_l) \pi_{j,0}(\cone{X}_k \wedge \cone{X}_l) + \sum_{1 \leq l \leq 4} \cone{g}(\cone{\nabla}_{X_1} \partial_r, \cone{X}_l) \pi_{j,0}\left(\cone{X}_4 \wedge \cone{X}_l\right)\\
  &= X_1 + \frac{1}{2} \sum_{1 \leq k,l \leq 3} g(\nabla_{X_1} X_k, X_l) \pi_{j,0}(\cone{X}_k \wedge \cone{X}_l) - \pi_{j,0}\left(\cone{X}_1 \wedge \cone{X}_4\right).
\end{align*}
Here, we used formulas of the Levi-Civita connection on the cone (see \cite{O'Neill}):
\[\cone{\nabla}_{X} Y = \nabla_X Y - r g(X,Y) \partial _r, \quad \cone{\nabla}_{X} \partial _r = \frac{1}{r} X.\]
Since the anti-self-dual 2-forms act on $W_{j,0}$ trivially, we have $\pi_{j,0}\left(\cone{X}_1 \wedge \cone{X}_4\right) = \pi_{j,0}\left(\cone{X}_2 \wedge \cone{X}_3\right)$.
Thus, for a spinor $\varphi \in \Gamma(S_j)$ on $M$, we obtain
\begin{align*}
  \cone{\nabla}_{X_1} \pi^{\ast} \varphi &= \pi^{\ast}\left(\nabla_{X_1} \varphi\right) - \pi_{j,0}\left(\cone{X}_1 \wedge \cone{X}_4\right) \pi^{\ast} \varphi\\
  &= \pi^{\ast}\left(\nabla_{X_1} \varphi\right) -  \pi^{\ast}(\pi_{j}\left(X_2 \wedge X_3\right)\varphi)\\
  &= \pi^{\ast}\left(\nabla_{X_1} \varphi - \pi_j(X_2 \wedge X_3)\varphi\right)\\
  &= \pi^{\ast}\left(\nabla_{X_1} \varphi - \frac{1}{2}\pi_j(X_1)\varphi\right).
\end{align*}
Here, we remark that the isomorphism $\su(2) \cong \so(3)$ is given by
\begin{equation}
  \label{eq:isom_su2_and_so3}
  \sigma_1 \mapsto 2(e_2 \wedge e_3),\quad \sigma_2 \mapsto 2(e_3 \wedge e_1),\quad \sigma_3 \mapsto 2(e_1 \wedge e_2).
\end{equation}
For $X_2, X_3$, we can show the similar formulas. 
Hence, we have $\cone{\nabla}_X \pi^{\ast} \varphi = \pi^{\ast}\left(\nabla_X \varphi - \frac{1}{2}\pi_j(X)\varphi\right)$ for $X \in TM$.
Also, we have $\cone{\nabla}_{\partial_r} \pi^{\ast} \varphi = 0$.
If we use a representation $(\pi_{0,j}, W_{0,j} = \C \boxtimes W_j)$ instead of $(\pi_{j,0}, W_{j,0})$, then $\cone{S}_{0,j} \cong \pi^{\ast} S_j$ and we have
\[\cone{\nabla}_X \pi^{\ast} \varphi = \pi^{\ast}\left(\nabla_X \varphi + \frac{1}{2}\pi_j(X)\varphi\right).\]
In summary, we have the following theorem.
\begin{theoremalpha}
  \label{theo:cone_construction}
  Let $(M,g)$ be a 3-dimensional Riemannian spin manifold and $(\cone{M}, \cone{g})$ the cone over $(M,g)$. Then, the following two are equivalent:
  \begin{enumerate}
    \item $M$ admits a spin $(j + \frac{1}{2})$ Killing spinor with Killing number $\mu = \frac{1}{2}$ (resp. $\mu = -\frac{1}{2}$).
    \item The cone $\cone{M}$ admits a parallel spinor on the bundle $\cone{S}_{j,0}$ with helicity $j + \tfrac{1}{2}$ (resp. $\cone{S}_{0,j}$ with $-(j + \tfrac{1}{2})$).
  \end{enumerate}
\end{theoremalpha}
\begin{proof}
  By the above argument, if $\varphi \in \Gamma(S_j)$ is a spin $(j + \frac{1}{2})$ Killing spinor with Killing number $\mu = \frac{1}{2}$, then $\pi^{\ast}\varphi \in \Gamma(\cone{S}_{j,0})$ is parallel with respect to $\cone{\nabla}$. 
  Conversely, if $\cone{\varphi} \in \Gamma(\cone{S}_{j,0})$ is a parallel spinor on $\cone{M}$, then
  \begin{align}
    \cone{\nabla}_{\partial_r} \cone{\varphi} &= \frac{\partial \cone{\varphi}}{\partial r} = 0 \label{eq:cone1}\\
    \cone{\nabla}_X \cone{\varphi} &= X\cone{\varphi} + \frac{1}{2} \sum_{1 \leq k,l \leq 3} g(\nabla_{X} X_k, X_l) \pi_{j,0}(\cone{X}_k \wedge \cone{X}_l)\cone{\varphi} - \pi_{j,0}\left(\cone{X} \wedge \cone{X}_4\right)\cone{\varphi} = 0 \quad (\forall X \in TM). \label{eq:cone2}
  \end{align}
  From equation \eqref{eq:cone1}, $\cone{\varphi}$ is independent of $r$ so that $\varphi := \cone{\varphi}|_{M} \in \Gamma(S_j)$ satisfies $\cone{\varphi} = \pi^{\ast}\varphi$. Then, by equation \eqref{eq:cone2}, we have
  \[\nabla_X \varphi - \frac{1}{2}\pi_j(X)\varphi = 0 \quad (\forall X \in TM),\]
  that is, $\varphi$ is a spin $(j + \frac{1}{2})$ Killing spinor with Killing number $\mu = \frac{1}{2}$.
  The other case can be shown in the same way.
\end{proof}

\subsection{Higher spin twistor spinors}
The spin $(j + \frac{1}{2})$ Killing spinors are in $\ker T^+_j \cap \ker T^-_j$. So we study spinors in $\ker T^+_j \cap \ker T^-_j$.

\begin{defi}
  A spinor $\varphi \in \Gamma(S_j)$ is called a \emph{spin $(j + \frac{1}{2})$ twistor spinor} or \emph{higher spin twistor spinor} (or simply \emph{twistor spinor}) if $\varphi$ is in $\ker T^+_j \cap \ker T^-_j$.
\end{defi}
This definition is a natural generalization of usual twistor spinors (i.e. $j = 0$ case).
We showed the projection $S_j \otimes TM$ onto $S_j$ is given by
\[\Pi_0(\varphi \otimes X) = p_j(X)\varphi = \frac{1}{\sqrt{(2j+1)(2j+3)}}\pi_j(X)\varphi.\]
So the projection onto $S_{j+1} \oplus S_{j-1}$ is given by
\[\Pi^{\perp}_0 (\varphi \otimes X) = \varphi \otimes X + \frac{1}{(2j+1)(2j+3)}\sum_i \pi_j(e_i)\pi_j(X)\varphi \otimes e_i.\]
Here, we note that the embedding $S_j \to S_j \otimes TM$ is given by \eqref{eq:embedding_of_W_j}.
Thus the condition that spinor $\varphi \in \Gamma(S_j)$ is in $\ker T^+_j \cap \ker T^-_j$ is equivalent to
\[0 = \Pi^{\perp}_0 \left(\sum_i\nabla_{e_i}\varphi \otimes e_i\right) = \sum_i \left(\nabla_{e_i}\varphi + \frac{1}{(2j+1)(2j+3)} \pi_j(e_i)D_j\varphi\right) \otimes e_i.\]
Now we have the following proposition.

\begin{prop}
  \label{prop:twistor_spinor}
  The following two are equivalent.
  \begin{enumerate}
    \item $\varphi \in \Gamma(S_j)$ is a spin $(j + \frac{1}{2})$ twistor spinor.
    \item $\varphi \in \Gamma(S_j)$ satisfies the twistor equation:
    \begin{equation} \label{eq:twistor equation}
    \nabla_X \varphi + \frac{1}{2j+3}\pi_j(X)D_j\varphi = 0 \quad \forall X \in TM.
    \end{equation}
  \end{enumerate}
  In particular, a spinor $\varphi \in \Gamma(S_j)$ is a Killing spinor if and only if $\varphi$ is an eigenspinor of $D_j$ and in $\ker T^+_j \cap \ker T^-_j$.
\end{prop}

Next, we calculate the upper bound of the dimension of the space of spin $(j + \frac{1}{2})$ twistor spinors. 

\begin{prop}
  On a 3-dimensional spin manifold (not necessarily compact),
  \begin{equation} \label{eq:dim_of_twistorspinors}
  \dim \ker T^+_j \cap \ker T^-_j \leq 2 \times (2j+2).
  \end{equation}
  This is a sharp estimate. In fact, $\R^3, \sphere{3}$ and $\hb{3}$ are examples satisfying the equality.
\end{prop}
First, let us construct twistor spinors on $\R^3$. If $u$ is a constant spinor on $\R^3$ with spin $(j + \frac{1}{2})$, then it is a parallel spinor and hence automatically a twistor spinor. Furthermore, let $(x^1,x^2,x^3)$ denote the standard coordinate on $\R^3$, and $(e_1,e_2,e_3)$ the standard orthonormal frame. For the vector field $x=x^1e_1+x^2e_2+x^3e_3$, it follows directly from the twistor equation \eqref{eq:twistor equation} that $\pi_j(x)u$ is another twistor spinor. Thus, the space of all twistor spinors on $\R^3$ can be explicitly expressed as 
\[\ker T^+_j \cap \ker T^-_j=\left\{ u+\pi_j(x)v\mid u,v \text{ is a constant spinor} \right\},\]
which gives the limiting case of the estimate \eqref{eq:dim_of_twistorspinors}.
We will show that the equality holds on $\sphere{3}$ in \S \ref{sec:higher_killing_on_S3} and $\hb{3}$ in \S \ref{sec:higher_killing_on_H3}.
\begin{proof}
  The following argument is a generalization of the one for the spin $\tfrac{1}{2}$ case (see \cite{BFGK}).
  For $\varphi \in \ker T^+_j \cap \ker T^-_j$, taking the covariant derivative of the twistor equation \eqref{eq:twistor equation} yields
  \[
  \nabla_Y\nabla_X\varphi+\dfrac{1}{2j+3}\pi_j(\nabla_YX)D_j\varphi+\dfrac{1}{2j+3}\pi_j(X)\nabla_Y(D_j\varphi)=0 \quad \forall X,Y\in TM.
  \]
  By taking the difference between this equation and the one obtained by interchanging the roles of $X$ and $Y$, we get
  \[R_j(X, Y)\varphi + \frac{1}{2j+3}(\pi_j(Y)\nabla_X(D_j\varphi) - \pi_j(X)\nabla_{Y}(D_j\varphi)) = 0.\]
  Putting $Y=e_i$, multiplying $\pi_j(e_i)$ to this equation from left, and summing over $i$, we have
  \begin{align}
    -(2j+1)\nabla_X(D_j \varphi) &= -\sum_{i=1}^3 \left(\pi_j(e_i)R_j(X, e_i)\varphi - \frac{1}{2j+3}\pi_j(e_i)\pi_j(X)\nabla_{e_i}(D_j\varphi)\right) \notag\\
    &= -\sum_{i=1}^3 \left(\pi_j(e_i)R_j(X, e_i)\varphi - \pi_j(X)q(R)\varphi - \frac{1}{2j+3}\pi_j([e_i,X])\nabla_{e_i}(D_j\varphi)\right). \label{eq:the dimension of twistor spinors1}
  \end{align}
  On the other hand, recalling equation \eqref{eq:comm_rel_of_pipm}, we see that
  \[
  \dfrac{1}{2(2j+1)}\pi_j(X)\pi_j(Y)+\dfrac{2j+1}{2}\pi_{j-1}^+(X)\pi_j^-(Y)=-\dfrac{1}{4}\pi_j([Y,X]_{\g})-\left(j+\dfrac{1}{2}\right)\langle X,Y\rangle \quad \forall X,Y\in TM.
  \]
  Putting $Y=e_i$, acting this equation on $\nabla_{e_i}(D_j\varphi)$, and summing over $i$, we have
  \[
  \dfrac{1}{2}\pi_j(X)D_j^2\varphi+\dfrac{2j+1}{2}\pi_{j-1}^+(X)T_j^-D_j\varphi=-\dfrac{1}{4}\sum_{i=1}^3\pi_j([e_i,X]_{\g})\nabla_{e_i}(D_j\varphi)-\left(j+\dfrac{1}{2}\right)\nabla_X(D_j\varphi).
  \]
  Noting $D_j^2\varphi=\dfrac{2j+3}{2j+1}q(R)\varphi$ from equation \eqref{eq:wetzenbock_formula} and taking into account formula \eqref{eq:twisted_weitzenbock_formula-}, we have a simplified expression
  \begin{equation} \label{eq:the dimension of twistor spinors2}
   \sum_{i=1}^{3} \pi_j([e_i,X]_{\g})\nabla_{e_i}(D_j\varphi) = -2(2j+1)\nabla_X(D_j\varphi) - 2\dfrac{2j+3}{2j+1}\pi_j(X)q(R)\varphi - 2(2j+1)\pi_{j-1}^+(X)q^-(R)\varphi.
  \end{equation}
  Combining the above two equations \eqref{eq:the dimension of twistor spinors1}, \eqref{eq:the dimension of twistor spinors2}, we get
  \[
  \nabla_X(D_j\varphi)=\dfrac{2j+3}{(2j+1)^2}\left\{-\dfrac{2j-1}{2j+1}\pi_j(X)q(R)+2\dfrac{2j+1}{2j+3}\pi_{j-1}^+(X)q^-(R)+\sum_{i=1}^3\pi_j(e_i)R_j(X,e_i)\right\}\varphi.
  \]
  Let us denote the right-hand side by $K(X) \in \Gamma(\End(S_j))$, and define a new covariant derivative on $\Gamma(S_j \oplus S_j)$ by
  \[
  \nabla^s_X = \left(
  \begin{array}{ccc}
  \nabla_X & \frac{1}{2j+3}\pi_j(X) \\[1ex]
  K(X) & \nabla_X
  \end{array}
  \right).
  \]
  Then, we know that $\varphi$ is in $\ker T^+_j \cap \ker T^-_j$ if and only if $(\varphi,D_j\varphi)\in\Gamma(S_j\oplus S_j)$ is parallel with respect to $\nabla^s$. Finally, since the bundle $S_j$ has rank $2j+2$, we arrive at the inequality.
\end{proof}

\subsection{Eigenvalue estimate for the higher spin Dirac operator}
\label{sec:eigenvalue_estimate}
In the usual spinor case ($j = 0$), there is a well-known eigenvalue estimate for the Dirac operator on compact Riemannian spin manifolds due to Th. Friedrich \cite{Fri80}.
Moreover, the limiting case of this estimate is characterized by the existence of Killing spinors.
Indeed, Killing spinors are eigenspinors of the Dirac operator attaining the limiting case.

We shall show a similar eigenvalue estimate for the higher spin Dirac operator $D_j$ on compact 3-dimensional Riemannian spin manifolds.
First, the space of the sections of $S_j$ on a compact Riemannian spin manifold $M$ is decomposed into the direct sum
\[\Gamma(S_j) = \ker (T_{j-1}^+)^{\ast} \oplus T_{j-1}^+(\Gamma(S_{j-1})) = \ker T_j^- \oplus T_{j-1}^+(\Gamma(S_{j-1})).\]
This follows from the fact that $T_{j-1}^+$ is an overdetermined elliptic operator since the principal symbol of $T_{j-1}^+$ is $\sigma_{\xi, x}(T_{j-1}^+) = \pi_{j-1}^+(\xi)$ (a more detailed dicussion can be found in \cite{HT}).

In the physical context, we often consider a massive Dirac equation $(\sldirac + M)\psi_{\mu_1 \cdots \mu_j} = 0$ with additional condition $\gamma^{\mu}\psi_{\mu \mu_2 \cdots \mu_j} = 0$. 
Here, $\mu_1 \cdots \mu_j$ are symmetric spacetime indices and each $\psi_{\mu_1 \cdots \mu_j}$ is a spinor field with spin $\frac{1}{2}$.
If $M \neq 0$, any solution $\psi$ satisfy $\psi^\mu_{\,\,\,\,\mu\mu_3 \cdots \mu_j} = 0$ and $\partial^{\mu} \psi_{\mu\mu_2 \cdots \mu_j} = 0$ (see \cite{RS41} and \cite{BennTucker87}). 
After normalizing the constant, the twisted Dirac operator is of the form (see Remark \ref{rema:twisted_Dirac})
\[\sldirac = \begin{pmatrix}
  D_j & T_{j-1}^+ \\
  T_{j}^- & D_{j-1}
\end{pmatrix},\]
and the additional condition $\gamma^{\mu}\psi_{\mu \mu_2 \cdots \mu_j} = 0$ means that spin $(j - \frac{1}{2})$ component of $\psi$ is zero, namely, $\psi = {}^t(\tilde{\psi}, 0) \in \Gamma(S_j \oplus S_{j-1})$ for some $\tilde{\psi} \in \Gamma(S_j)$.
Then, the massive Dirac equation reduces to
\[(D_j + M)\tilde{\psi} = 0, \quad T_j^- \tilde{\psi} = 0.\]
Thus, it is physically meaningful to study the eigenvalue problem of $D_j$ on $\ker T_j^-$.
\begin{theoremalpha}
  \label{theo:eigenvalue_estimate}
  Let $(M,g)$ be a compact 3-dimensional Riemannian spin manifold and 
  \begin{align*}
    r_0 := \min_M q(R) :=& \min \{d \in \R \mid x \in M, d \text{ is an eigenvalue of } q(R)_x \}\\
    =& \max \{d \in \R \mid \forall x \in M, q(R)_x \geq d\}
  \end{align*}
  Then, the first eigenvalue $\lambda^2$ of $D_j^2$ on $\ker T_j^-$ satisfies
  \[\lambda^2 \geq \frac{2j+3}{2j+1}r_0.\]
  The equality holds if and only if there exists a spin $(j + \frac{1}{2})$ Killing spinor.
\end{theoremalpha}
\begin{proof}
  For $\varphi \in \ker T_j^-$, the Weitzenb\"ock formula \eqref{eq:wetzenbock_formula} implies
  \[\|D_j\varphi\|^2_{L^2} = \frac{2j+3}{2j+1} ( q(R)\varphi, \varphi)_{L^2} + \frac{(2j+3)^2}{8(j+1)} \|T_j^+ \varphi\|^2_{L^2} \geq \frac{2j+3}{2j+1}r_0\|\varphi\|^2_{L^2}.\]
  Hence, the first eigenvalue $\lambda^2$ of $D_j^2$ on $\ker T_j^-$ satisfies
  \[\lambda^2 \geq \frac{2j+3}{2j+1}r_0.\]
  
  If the equality holds, then the above inequality shows that $\varphi \in \ker T_j^- \cap \ker T_j^+$, that is, $\varphi$ is a twistor spinor. Hence from Proposition \ref{prop:twistor_spinor}, $\varphi$ is a Killing spinor. 

  Conversely, if there exists a spin $(j + \frac{1}{2})$ Killing spinor $\varphi$ with Killing number $\mu$, then $M$ is an Einstein manifold by Theorem \ref{theo:3dim_Killing_and_Einstein} and Proposition \ref{prop:expression_of_q(R)} says that a curvature action $q(R)$ is constant such that
  \[q(R) = \frac{\scal}{24}(2j+3)(2j+1).\]
  Thus, $r_0 = \frac{\scal}{24}(2j+3)(2j+1)$ and we have
  \[D_j^2 \varphi = (2j+3)^2\mu^2 \varphi =  (2j+3)^2 \frac{\scal}{24} \varphi = \frac{2j+3}{2j+1}r_0 \varphi.\]
\end{proof}

\subsection{Higher spin Killing spinors on $\sphere{3}$}
\label{sec:higher_killing_on_S3}
Since the scalar curvature of $\sphere{3}$ is $\scal = 6$, higher spin Killing spinors on $\sphere{3}$ are of Killing number $\mu = \pm 1/2$ by Theorem \ref{theo:3dim_Killing_and_Einstein}.
First, we see that spin $(j+\frac{1}{2})$ spinor bundle $S_j$ on $\sphere{3}$ can be trivialized by spin $(j+\frac{1}{2})$ Killing spinors.

\begin{prop}
  \label{prop:triv_Sj_on_S3}
  The spin $(j + \frac{1}{2})$ spinor bundle $S_j$ on $\sphere{3}$ can be trivialized by the spin $(j + \frac{1}{2})$ Killing spinors for $\mu = \frac{1}{2}$ as well as for $\mu = -\frac{1}{2}$. In particular, the dimension of the space of spin $(j + \frac{1}{2})$ Killing spinors is $2j + 2$ for each $\mu = \pm\frac{1}{2}$.
  Thus, the dimension of the space of spin $(j + \tfrac{1}{2})$ twistor spinors is $2(2j + 2)$.
\end{prop}
\begin{proof}
  We set the modified connection $\widetilde{\nabla}_X = \nabla_X - \mu\pi_j(X)$ for $\mu = \pm\frac{1}{2}$. Since the Riemannian curvature tensor $R$ on $\sphere{3}$ is given by
  \[R(X, Y)Z = g(Y,Z)X - g(X,Z)Y \quad (\text{for } X,Y,Z \in TM),\]
  the curvature tensor $R_j$ on $S_j$ is given by
  \[R_j(X, Y) = -\frac{1}{4}\pi_j([X,Y]_{\g}).\]
  Therefore curvature tensor $R^{\widetilde{\nabla}}$ with respect to the modified connection $\widetilde{\nabla}$ satisfies
  \[R^{\widetilde{\nabla}}(X, Y) = R_j(X,Y) + \frac{1}{4}\pi_j([X,Y]_{\g}) = 0,\]
  that is, the modified connection $\widetilde{\nabla}$ is a flat connection.
  One can see more detailed calculation in \cite[Lemma 1]{Bar96}.
\end{proof}

We give higher spin Killing spinors on $\sphere{3}$ explicitly.
It is well-known that Lie group $\sphere{3}$ carries an orthonormal frame of left-invariant Killing vector fields $\{\xi_1, \xi_2, \xi_3\}$ satisfying
\[\nabla_{\xi_1} \xi_2 = -\nabla_{\xi_2} \xi_1 = \xi_3, \quad  \nabla_{\xi_2} \xi_3 = -\nabla_{\xi_3} \xi_2 = \xi_1, \quad \nabla_{\xi_3} \xi_1 = -\nabla_{\xi_1} \xi_3 = \xi_2.\]
Hence for all left-invariant vector fields $\xi$ and for all vector fields $X$, we have
\[\nabla_X \xi = \ast(X \wedge \xi) = \frac{1}{2}[X,\xi]_{\g},\]
where the second equality follows from \eqref{eq:g-bracket}.
For left-invariant vector fields $\xi$ and spin $(j+\frac{1}{2})$ Killing spinor $\varphi \in \Gamma(S_j)$ with $\mu = \frac{1}{2}$,
\[\psi := \pi^+_j(\xi)\varphi \quad \in \Gamma(S_{j+1})\]
is a spin $(j + \frac{3}{2})$ Killing spinor with $\mu = \frac{1}{2}$.
Indeed, $\psi$ satisfies
\begin{align*}
  \nabla_X \psi &= \pi^+_j(\nabla_X \xi)\varphi + \frac{1}{2}\pi^+_j(\xi)\pi_j(X)\varphi\\
  &= \frac{1}{2}\pi^+_j([X,\xi]_{\g})\varphi + \frac{1}{2}\pi_{j+1}(X)\pi^+_j(\xi)\varphi - \frac{1}{2}\pi^+_j([X,\xi]_{\g})\varphi\\
  &= \frac{1}{2}\pi_{j+1}(X)\psi.
\end{align*}
We need to check that $\psi$ is not identically zero.
By equation \eqref{eq:comm_rel_of_pipm}, we have
\begin{align*}
  \|\psi\|^2 &= - \left\langle\pi^-_{j+1}(\xi)\pi^+_j(\xi)\varphi, \varphi\right\rangle\\
  &= \frac{1}{(2j+3)^2}\left\langle\pi_j(\xi)\pi_j(\xi)\varphi, \varphi\right\rangle + \left\langle\xi, \xi\right\rangle \left\langle\varphi, \varphi\right\rangle\\
  &= -\frac{1}{(2j+3)^2}\|\pi_j(\xi)\varphi\|^2 + \|\xi\|^2\|\varphi\|^2\\
  & \geq -\frac{(2j+1)(2j+3)}{(2j+3)^2}\|\xi\|^2\|\varphi\|^2 + \|\xi\|^2\|\varphi\|^2\\
  &= \frac{2}{2j+3}\|\xi\|^2\|\varphi\|^2 > 0,
\end{align*}
where the inequality follows from the following inequality:
\[\|\pi_j(\xi)\varphi\|^2 = \left\|\sqrt{(2j+1)(2j+3)}\Pi_j(\varphi \otimes \xi)\right\|^2 \leq (2j+1)(2j+3)\|\xi\|^2\|\varphi\|^2.\]
Similarly, for right-invariant vector fields $\xi$ and spin $(j+\frac{1}{2})$ Killing spinor $\varphi \in \Gamma(S_j)$ with $\mu = -\frac{1}{2}$, $\psi := \pi^+_j(\xi)\varphi$ is a spin $(j + \frac{3}{2})$ Killing spinor with $\mu = -\frac{1}{2}$.
Now we have following theorem.
\begin{theo}
  \label{theo:Killing_spinor_on_S3}
  On $\sphere{3}$, for any non-zero left-invariant (resp. right-invariant) vector fields $\xi$ and any spin $(j+\frac{1}{2})$ Killing spinors $\varphi \in \Gamma(S_j)$ with Killing number $\mu = \frac{1}{2}$ (resp. $\mu = -\frac{1}{2}$), $\pi_j^+(\xi)\varphi \in \Gamma(S_{j+1})$ is a spin $(j + \frac{3}{2})$ Killing spinor with Killing number $\mu = \frac{1}{2}$ (resp. $\mu = -\frac{1}{2}$).
  Moreover, all Killing spinors on $\sphere{3}$ are obtained inductively by this construction from lower spin Killing spinors.
\end{theo}

In order to prove this theorem, we need some preparations.
It is well-known that $G := \SU(2) \times \SU(2)$ acts on $\sphere{3} \cong \SU(2)$ transitively by 
  \[(g_1, g_2) \cdot p = g_1 p g_2^{-1} \quad\quad (g_1,g_2) \in G, \,\,p \in \sphere{3} = \SU(2),\]
  and the isotropy subgroup at the identity $e \in \sphere{3}$ is the diagonal subgroup $H := \{(g,g) \mid g \in \SU(2)\} \cong \SU(2)$. 
  Thus, we have $\sphere{3} \cong G / H$.
  By this expression, the projection $G \to G / H$ can be seen as the principal $\SU(2)$-bundle over $\sphere{3}$, which is the spin structure on $\sphere{3}$.
  So the spin $(j + \frac{1}{2})$ spinor bundle $S_j$ over $\sphere{3}$ is given by the associated bundle $S_j = G \times_{H} W_j$, and the space of sections $\Gamma(S_j)$ is identified with the space of $H$-equivariant smooth functions
  \begin{align*}
    \Gamma(S_j) &\cong C^{\infty}(G, W_j)^{H}\\
    & = \{\varphi \in C^{\infty}(G, W_j) \mid \varphi(g_1h,g_2h) = \pi_j(h^{-1})\varphi(g_1,g_2),\quad \forall (g_1,g_2) \in G, \forall h \in \SU(2)\}.
  \end{align*}
  Hence, $\Gamma(S_j)$ can be seen as a representation space of $G$ by
  \[((g_1', g_2') \cdot \varphi)(g_1, g_2) := \varphi((g_1')^{-1}g_1, (g_2')^{-1}g_2) \quad \forall (g_1', g_2'),  \forall (g_1,g_2) \in G.\]
  Now we set two subgroups of $G=\SU(2) \times \SU(2)$:
\[G_L := \SU(2) \times \{e\}, \quad G_R := \{e\} \times \SU(2).\]
For any element $\psi \in W_j$, we define a map $\psi^L \colon G \to W_j$ as $G_L$-invariant and $H$-equivariant.
That is, 
\[\psi^L(g_1,g_2) = \psi^L(g_2,g_2) = \pi_j(g_2^{-1})\psi^L(e,e) = \pi_j(g_2^{-1})\psi.\]
Since $\psi^L$ is $H$-equivariant, it defines a section of $S_j$.
We denote this section by the same symbol $\psi^L \in \Gamma(S_j)$.
And we can define $\psi^R \in \Gamma(S_j)$ similarly as $G_R$-invariant and $H$-equivariant:
\[\psi^R(g_1,g_2) = \psi^R(g_1,g_1) = \pi_j(g_1^{-1})\psi^R(e,e) = \pi_j(g_1^{-1})\psi.\]
We set two subspaces of $\Gamma(S_j)$:
\[K_{2j+1}^{+} := \{\psi^L \in \Gamma(S_j) \mid \psi \in W_j\}, \quad K_{2j+1}^{-} := \{\psi^R \in \Gamma(S_j) \mid \psi \in W_j\}.\]
By the definition of these spaces, $K_{2j+1}^{+}$ is isomorphic to an irreducible representation of $G$ with highest weight $(0, 2j+1)$.
Indeed, $\C \boxtimes W_j \ni \psi \mapsto \psi^L \in K_{2j+1}^{+}$ is a $G$-equivariant isomorphism.
Similarly, $K_{2j+1}^{-}$ is isomorphic to an irreducible representation of $G$ with highest weight $(2j+1, 0)$.
\begin{lemm}
  \label{lemm:irrep_and_higherspin_killing}
  For any $\psi \in W_j$, $\psi^L, \psi^R \in \Gamma(S_j)$ satisfy
  \[\nabla_X \psi^L = \frac{1}{2}\pi_j(X)\psi^L, \quad \nabla_X \psi^R = -\frac{1}{2}\pi_j(X)\psi^R \quad (\forall X \in \Gamma(T\sphere{3})).\]
  Therefore, $K_{2j+1}^{+}$ (resp. $K_{2j+1}^{-}$) is the space of spin $(j + \frac{1}{2})$ Killing spinors with Killing number $\mu = \frac{1}{2}$ (resp. $\mu = -\frac{1}{2}$).
\end{lemm}
\begin{proof}
  We prove only the first equation.
  It is sufficient to show that the equation holds for left-invariant vector fields $X$ on $\sphere{3}$, since both sides are tensorial in $X$.
  Let $\g = \su(2) \oplus \su(2)$ be the Lie algebra of $G$, and $\h = \{(X,X) \mid X \in \su(2)\}$ the Lie algebra of $H$.
  We set $\m := \{(X,-X) \mid X \in \su(2)\}$, which is identified with the tangent space $T_{[e]} \sphere{3}$ at the identity coset $[e] \in G/H$.
  For $X \in \m$, we denote by $\tilde{X}$ the left-invariant vector field on $\sphere{3}$ defined by $X$.
  On the other hand, since $\sphere{3} \cong \SU(2)$, $T_{[e]}\sphere{3}$ can be identified with $\su(2)$ by the map $(X,-X) \mapsto 2X$.
  Thus, for a left-invariant vector field $\tilde{X} \in \Gamma(\sphere{3})$, we have
  \begin{align*}
  (\nabla_{\tilde{X}} \psi^L)_{g_1g_2^{-1}} &= \frac{d}{dt}\Big|_{t=0} \psi^L(g_1\exp(t\tfrac{X}{2}), g_2\exp(-t\tfrac{X}{2}))\\
  &= \frac{d}{dt}\Big|_{t=0} \pi_j(\exp(t\tfrac{X}{2})g_2^{-1})\psi^L(e,e)\\
  &= \frac{d}{dt}\Big|_{t=0} \pi_j(\exp(t\tfrac{X}{2}))\psi^L(g_1,g_2)\\
  &= \frac{1}{2}\pi_j(\tilde{X})\psi^L_{g_1g_2^{-1}}.
  \end{align*}
  The other equation can be shown similarly.
  Therefore, $\psi^L$ (resp. $\psi^R$) is a spin $(j + \frac{1}{2})$ Killing spinor with Killing number $\mu = \frac{1}{2}$ (resp. $\mu = -\frac{1}{2}$).
  By comparing the dimensions, we obtain the last statement.
\end{proof}

\begin{proof}[Proof of Theorem \ref{theo:Killing_spinor_on_S3}]
  We have already shown the first statement.
  It is sufficient to show the second statement.
  We write $K_2^+ \subset \Gamma(TM^{\C})$ for the irreducible component of $G$ with highest weight $(0,2)$.
  Similarly, we write $K_2^- \subset \Gamma(TM^{\C})$ for the irreducible component of $G$ with highest weight $(2,0)$.
  By the same argument as Lemma \ref{lemm:irrep_and_higherspin_killing}, $K_2^+$ (resp. $K_2^-$) is the space of left-invariant (resp. right-invariant) vector fields on $\sphere{3}$ (see also \S\ref{sec:Killing_tensor_on_S3}).
  We consider the map
  \[\Phi_j^{\pm} \colon K_2^{\pm} \otimes K_{2j + 1}^{\pm} \to K_{2j + 3}^{\pm}, \quad \xi \otimes \varphi \mapsto \pi_j^+(\xi)\varphi.\]
  By the first statement of this theorem, this map is well-defined and non-zero.
  Moreover, it is easy to see that $\Phi_j^{\pm}$ is $G$-equivariant.
  Therefore $\Phi_j^{\pm}$ must be surjective by Schur's lemma. This implies that all higher spin Killing spinors on $\sphere{3}$ are obtained from lower spin Killing spinors.
\end{proof}

Let us give an explicit formula of higher spin Killing spinors on $\sphere{3}$.
We trivialize the spin structure $G \to \sphere{3}$ by a global section $\sphere{3} \ni p \mapsto (p, e) \in G$.
The spin $(j + \frac{1}{2})$ spinor bundle $S_j$ on $\sphere{3}$ is trivialized as
\[S_j = G \times_H W_j \xrightarrow{\sim} \sphere{3} \times W_j, \quad [(g_1, g_2), \varphi] \mapsto (g_1g_2^{-1}, \pi_j(g_2)\varphi).\]
This trivialization is the same as the trivialization by the spin $(j + \frac{1}{2})$ Killing spinors with $\mu = \frac{1}{2}$ (Proposition \ref{prop:triv_Sj_on_S3}).
Indeed, under this trivialization, $\psi^L \in \Gamma(S_j)$ is expressed as
\[[(g_1, g_2), \psi^L(g_1, g_2)] \mapsto (g_1g_2^{-1}, \pi_j(g_2)\psi^L(g_1, g_2)) = (g_1g_2^{-1}, \psi).\]
Namely, the spin $(j + \frac{1}{2})$ Killing spinors with $\mu = \frac{1}{2}$ are the constant sections of $\sphere{3} \times W_j$.
On the other hand, $\psi^R$ is expressed as
\[[(g_1, g_2), \psi^R(g_1, g_2)] \mapsto (g_1g_2^{-1}, \pi_j(g_2)\psi^R(g_1, g_2)) = (g_1g_2^{-1}, \pi_j(g_2g_1^{-1})\psi).\]
Therefore, if we write $g_1g_2^{-1} = x \in \sphere{3}$, the spin $(j + \frac{1}{2})$ Killing spinor with Killing number $\mu = -\frac{1}{2}$ is expressed as $x \mapsto \pi_j(x^{-1})\psi$ for some $\psi \in W_j$.
Summarizing this argument, we obtain the following proposition.
\begin{prop}
  Under the above trivialization of $S_j$ on $\sphere{3}$, the spin $(j + \frac{1}{2})$ Killing spinors are given as follows:
  \begin{enumerate}
    \item The spin $(j + \frac{1}{2})$ Killing spinor with Killing number $\mu = \frac{1}{2}$ is the constant section $x \mapsto \psi$, where $\psi \in W_j$.
    \item The spin $(j + \frac{1}{2})$ Killing spinor with Killing number $\mu = -\frac{1}{2}$ is $x \mapsto \pi_j(x^{-1})\psi$, where $\psi \in W_j$.
  \end{enumerate}
\end{prop}

\subsection{Higher spin Killing spinors on $\hb{3}$}
\label{sec:higher_killing_on_H3}
Since the scalar curvature of $\hb{3}$ is $\scal = -6$, higher spin Killing spinors on $\hb{3}$ are of Killing number $\mu = \pm i/2$ by Theorem \ref{theo:3dim_Killing_and_Einstein}.
In the spin $\frac{1}{2}$ case, explicit formulas of Killing spinors on $\hb{3}$ are obtained by Y. Fujii and K. Yamagishi in \cite{FujiiYamagishi86}. 
We generalize their result to higher spin cases.
We use the upper half-space model of $\hb{3}$:
\[\hb{3} = \{(x^1,x^2,x^3) \in \R^3 \mid x^1 > 0\}, \quad g = \frac{1}{(x^1)^2}((dx^1)^2 + (dx^2)^2 + (dx^3)^2).\]
An orthonormal frame $\{e_1, e_2, e_3\}$ on $\hb{3}$ is given by
\[e_1 = x^1 \frac{\partial}{\partial x^1}, \quad e_2 = x^1 \frac{\partial}{\partial x^2}, \quad e_3 = x^1 \frac{\partial}{\partial x^3}.\]
The Levi-Civita connection $\nabla$ on $\hb{3}$ is expressed as
\begin{align}
  \nabla_{e_1} e_1 &= 0, & \nabla_{e_1} e_2 &= -e_2, & \nabla_{e_1} e_3 &= -e_3, \notag\\
  \nabla_{e_2} e_1 &= -e_2, & \nabla_{e_2} e_2 &= e_1, & \nabla_{e_2} e_3 &= 0, \label{eq:LCconn_on_H3}\\
  \nabla_{e_3} e_1 &= -e_3, & \nabla_{e_3} e_2 &= 0, & \nabla_{e_3} e_3 &= e_1. \notag
\end{align}
\begin{lemm}
  \label{lemm:conn_on_Sk_on_H3}
  The induced connection $\nabla$ on the spin $(j + \frac{1}{2})$ spinor bundle $S_j$ on $\hb{3}$ is expressed as
  \[\nabla_{e_1} = e_1, \quad \nabla_{e_2} = e_2 - \frac{1}{2}\pi_j(e_3), \quad \nabla_{e_3} = e_3 + \frac{1}{2}\pi_j(e_2).\]
\end{lemm}
\begin{proof}
  The induced connection on $S_j$ is given by
  \[\nabla_{X} = X + \sum_{1 \leq k < l \leq 3} g(\nabla_{X} e_k, e_l)\pi_j(e_k \wedge e_l).\]
  Thus, by \eqref{eq:LCconn_on_H3} and the isomorphism \eqref{eq:isom_su2_and_so3} between $\su(2)$ and $\so(3)$ , we have the desired expression.
\end{proof}

By this lemma, the Killing spinor equation with Killing number $\mu$ on $\hb{3}$ reduces to the following system of differential equations:
\begin{align*}
  \nabla_{e_1} \varphi &= e_1 \varphi = \mu \pi_j(e_1) \varphi,\\
  \nabla_{e_2} \varphi &= \left(e_2 - \frac{1}{2}\pi_j(e_3)\right)\varphi = \mu \pi_j(e_2) \varphi,\\
  \nabla_{e_3} \varphi &= \left(e_3 + \frac{1}{2}\pi_j(e_2)\right)\varphi = \mu \pi_j(e_3) \varphi.
\end{align*}
From Theorem \ref{theo:3dim_Killing_and_Einstein}, Killing number must be $\mu = \pm \frac{i}{2}$ on $\hb{3}$.
Thus, we should solve the following system of differential equations:
\begin{align}
  \frac{\partial \varphi}{\partial x^1} &= \pm \frac{i}{2x^1} \pi_j(e_1) \varphi, \notag\\
  \frac{\partial \varphi}{\partial x^2} &= \left(\pm \frac{i}{2x^1} \pi_j(e_2) + \frac{1}{2x^1}\pi_j(e_3)\right) \varphi, \label{eq:Killingeq_on_H3_1}\\
  \frac{\partial \varphi}{\partial x^3} &= \left(- \frac{1}{2x^1}\pi_j(e_2) \pm \frac{i}{2x^1} \pi_j(e_3)\right) \varphi. \notag
\end{align}
  Let $H, E, F \in \sl(2, \C)$ be the standard basis defined by
  \[H = \begin{pmatrix}1 & 0 \\ 0 & -1\end{pmatrix}, \quad E = \begin{pmatrix}0 & 1 \\ 0 & 0\end{pmatrix}, \quad F = \begin{pmatrix}0 & 0 \\ 1 & 0\end{pmatrix}.\]
  Since $H = -i\sigma_1, E = \frac{1}{2}(\sigma_2 - i\sigma_3), F = -\frac{1}{2}(\sigma_2 + i\sigma_3)$, the above differential equations \eqref{eq:Killingeq_on_H3_1} are rewritten as
  \begin{align*}
    \frac{\partial \varphi}{\partial x^1} &= -\frac{1}{2 x^1} \pi_j(H) \varphi, &\frac{\partial \varphi}{\partial x^2} &= \frac{i}{x^1}\pi_j(E)\varphi, &\frac{\partial \varphi}{\partial x^3} &= -\frac{1}{x^1}\pi_j(E)\varphi &\text{ for } \mu = \frac{i}{2},\\
    \frac{\partial \varphi}{\partial x^1} &= \frac{1}{2 x^1} \pi_j(H) \varphi, &\frac{\partial \varphi}{\partial x^2} &= \frac{i}{x^1}\pi_j(F)\varphi, &\frac{\partial \varphi}{\partial x^3} &= \frac{1}{x^1}\pi_j(F)\varphi &\text{ for } \mu = -\frac{i}{2}.
  \end{align*}
In order to solve these differential equations, we set $x = x^1, z = x^2 + i x^3$.
Then, the above differential equations are equivalent to
\begin{align}
  \frac{\partial \varphi}{\partial x} &= -\frac{1}{2 x} \pi_j(H) \varphi, &\frac{\partial \varphi}{\partial z} &= \frac{i}{x} \pi_j(E) \varphi, &\frac{\partial \varphi}{\partial \bar{z}} &= 0 &\text{ for } \mu = \frac{i}{2},\label{eq:killingeq_on_H3_2}\\
  \frac{\partial \varphi}{\partial x} &= \frac{1}{2 x} \pi_j(H) \varphi, &\frac{\partial \varphi}{\partial z} &= 0, &\frac{\partial \varphi}{\partial \bar{z}} &= \frac{i}{x} \pi_j(F) \varphi  &\text{ for } \mu = -\frac{i}{2}. \notag
\end{align}
Solving these differential equations, we have the following theorem.
\begin{theo}
  \label{theo:Killing_spinor_on_H3}
  The 3-dimensional hyperbolic space $\hb{3}$ admits spin $(j + \frac{1}{2})$ Killing spinors for $\mu = \pm \frac{i}{2}$. The dimension of the space of spin $(j + \frac{1}{2})$ Killing spinors is $2j + 2$ for each $\mu = \pm \frac{i}{2}$. Moreover, higher spin Killing spinors on $\hb{3}$ can be expressed explicitly.
  Thus, the dimension of the space of spin $(j + \tfrac{1}{2})$ twistor spinors is $2(2j + 2)$.
\end{theo}
\begin{proof}
  We only show the case of spin $\tfrac{3}{2}$ and $\mu = \frac{i}{2}$.
  The other cases can be shown similarly.
  We take the standard basis of $W_1$ as
  \[\pi_1(H) = \begin{pmatrix}
    3 & 0 & 0 & 0\\
    0 & 1 & 0 & 0\\
    0 & 0 & -1 & 0\\
    0 & 0 & 0 & -3
  \end{pmatrix}, \quad \pi_1(E) = \begin{pmatrix}
    0 & 3 & 0 & 0\\
    0 & 0 & 4 & 0\\
    0 & 0 & 0 & 3\\
    0 & 0 & 0 & 0
  \end{pmatrix}.\]
  Let $\varphi = {}^t(\varphi_1, \varphi_2, \varphi_3, \varphi_4) \in \Gamma(S_1)$ be a spin $\frac{3}{2}$ spinor.
  From \eqref{eq:killingeq_on_H3_2}, each $\varphi_i$ is independent of $\bar{z}$ and we solve the following system of differential equations:
  \begin{align}
    \frac{\partial \varphi_1}{\partial x} &= -\frac{3}{2x}\varphi_1, & \frac{\partial \varphi_1}{\partial z} &= \frac{3i}{x}\varphi_2, \label{eq:1}\\
    \frac{\partial \varphi_2}{\partial x} &= -\frac{1}{2x}\varphi_2, & \frac{\partial \varphi_2}{\partial z} &= \frac{4i}{x}\varphi_3, \label{eq:2}\\
    \frac{\partial \varphi_3}{\partial x} &= \frac{1}{2x}\varphi_3, & \frac{\partial \varphi_3}{\partial z} &= \frac{3i}{x}\varphi_4, \label{eq:3}\\
    \frac{\partial \varphi_4}{\partial x} &= \frac{3}{2x}\varphi_4, & \frac{\partial \varphi_4}{\partial z} &= 0. \label{eq:4}
  \end{align}
  From \eqref{eq:4}, $\varphi_4$ is independent of $z$ and is solved as
  \[\varphi_4(x,z) = C_4x^{\frac{3}{2}} \quad (C_4 \in \C).\]
  From the first equation of \eqref{eq:3}, $\varphi_3$ is of the form $\varphi_3(x,z) = A_3(z)x^{\frac{1}{2}}$ for some function $A_3(z)$.
  Substituting this into the second equation of \eqref{eq:3} and combining with $\varphi_4$, $A_3$ satisfies $\frac{d A_3}{d z} = 3i C_4$, and thus we have
  \[\varphi_3(x,z) = (C_3 + 3i C_4 z)x^{\frac{1}{2}} \quad (C_3 \in \C).\]
  Similarly, from the first equation of \eqref{eq:2}, $\varphi_2$ is of the form $\varphi_2(x,z) = A_2(z)x^{-\frac{1}{2}}$ for some function $A_2(z)$.
  Substituting this into the second equation of \eqref{eq:2} and combining with $\varphi_3$, $A_2$ satisfies $\frac{d A_2}{d z} = 4i(C_3 + 3i C_4 z)$, and thus we have
  \[\varphi_2(x,z) = \left(C_2 + 4i C_3 z - 6 C_4 z^2\right)x^{-\frac{1}{2}} \quad (C_2 \in \C).\]
  Finally, from the first equation of \eqref{eq:1}, $\varphi_1$ is of the form $\varphi_1(x,z) = A_1(z)x^{-\frac{3}{2}}$ for some function $A_1(z)$.
  Substituting this into the second equation of \eqref{eq:1} and combining with $\varphi_2$, $A_1$ satisfies $\frac{d A_1}{d z} = 3i(C_2 + 4i C_3 z - 6 C_4 z^2)$, and thus we have
  \[\varphi_1(x,z) = \left(C_1 + 3i C_2 z - 6 C_3 z^2 - 6i C_4 z^3\right)x^{-\frac{3}{2}} \quad (C_1 \in \C).\]
  Therefore, all spin $\frac{3}{2}$ Killing spinors with $\mu = \frac{i}{2}$ are given by
  \[\varphi(x,z) = C_1\begin{pmatrix}
    x^{-\frac{3}{2}}\\
    0\\
    0\\
    0
  \end{pmatrix} + C_2\begin{pmatrix}
    3izx^{-\frac{3}{2}}\\
    x^{-\frac{1}{2}}\\
    0\\
    0
  \end{pmatrix} + C_3\begin{pmatrix}
    -6z^2x^{-\frac{3}{2}}\\
    4izx^{-\frac{1}{2}}\\
    x^{\frac{1}{2}}\\
    0
  \end{pmatrix} + C_4\begin{pmatrix}
    -6iz^3x^{-\frac{3}{2}}\\
    -6z^2x^{-\frac{1}{2}}\\
    3izx^{\frac{1}{2}}\\
    x^{\frac{3}{2}}
  \end{pmatrix} \quad (C_1, C_2, C_3, C_4 \in \C).\]
  Thus, the dimension of the space of spin $\frac{3}{2}$ Killing spinors with $\mu = \frac{i}{2}$ is 4.
\end{proof}

\begin{rema}
  In general, each component of higher spin Killing spinors on $\hb{3}$ can be expressed as a product of a polynomial in $z$ and a power of $x$ if $\mu = \frac{i}{2}$. Similarly, if $\mu = -\frac{i}{2}$, each component can be expressed as a product of a polynomial in $\bar{z}$ and a power of $x$.
\end{rema}

\section{Killing spinor-type equation with integral spin}
\label{sec:integral_spin}
We have studied a differential equation $\nabla_X \psi = \mu \pi_j(X) \psi$ for half-integral spins.
Analogously, we can consider the case of integral spins.
In this section, we study the above differential equation on integral spin bundles and compare it with the half-integral spin case.
We focus on the 3-dimensional manifolds as in the previous sections.

\subsection{Ingredients of the integral spin bundles}
Let $(\rho_j, V_j)$ be the spin $j$ representation of $\Spin(3) \cong \SU(2)$ for $j \in \Z_{\geq 0}$, which reduces to the representation of $\SO(3)$.
In other words, $V_j$ is an irreducible representation of $\SU(2)$ with highest weight $2j$.
For example, $V_1 \cong \su(2) \otimes \C$ is the adjoint representation, which is equivalent to the natural representation of $\SO(3)$. In general, $V_j$ is identified with the $j$-th traceless symmetric tensor product of $\C^3$, that is, $V_j \cong \Sym^j_0 \C^3$. Thus, the spin $j$ bundle is a vector bundle $\Sym^j_0 := \Sym^j_0 TM^{\C}$ and its sections are traceless symmetric tensor fields of degree $j$.

Most of the arguments in \S \ref{sec:ingredients_of_higherspin_on_3dim} carry over to the integral spin case.
By the Clebsch-Gordan formula, we have
\[\Sym^j_0 \otimes T M^{\C} \cong \Sym^{j+1}_0 \oplus \Sym^j_0 \oplus \Sym^{j-1}_0.\]
Accordingly, we have three natural first order differential operators
\[D_j \colon \Gamma(\Sym^j_0) \to \Gamma(\Sym^j_0), \quad T_j^+ \colon \Gamma(\Sym^j_0) \to \Gamma(\Sym^{j+1}_0), \quad T_j^- \colon \Gamma(\Sym^j_0) \to \Gamma(\Sym^{j-1}_0),\]
and Clifford homomorphisms
\[\rho_j(X) \colon \Sym^j_0 \to \Sym^j_0, \quad \rho^+_j(X) \colon \Sym^j_0 \to \Sym^{j+1}_0, \quad \rho_j^-(X) \colon \Sym^j_0 \to \Sym^{j-1}_0, \quad \text{for } X \in TM.\]
Furthermore, the Weitzenb\"ock-type formulas also hold for these operators, see \cite{HT}.

\begin{rema}
  Unlike in the half-integral spin case, the operator $D_j$ is not elliptic.
\end{rema}

\begin{exam}
  When $j = 1$, the bundle $\Sym^1_0$ is the complexified tangent bundle $TM^{\C}$, which is isomorphic to cotangent bundle $T^{\ast}M^{\C}$.
  The operators $D_1, T_0^+, T_1^-$ are identified with well-known differential operators (up to constant):
  \begin{align*}
    D_1 &= \ast d = \operatorname{rot} \colon \Gamma(T^{\ast}M^{\C}) \to \Gamma(T^{\ast}M^{\C}),\\
    T_0^+ &= d = \operatorname{grad} \colon C^{\infty}(M, \C) \to \Gamma(T^{\ast}M^{\C}),\\
    T_1^- &= \delta = \operatorname{div} \colon \Gamma(T^{\ast}M^{\C}) \to C^{\infty}(M, \C).
  \end{align*}
\end{exam}

\subsection{Killing spinor equation on integral spin bundle}
\label{sec:Killing_spinor_type_eq_on_int_spin}
From now on, we consider the differential equation
\begin{equation}
  \nabla_X K = \mu \rho_j(X) K \quad (\forall X \in TM) \label{eq:int_spin_eq}
\end{equation}
for sections $K \in \Gamma(\Sym^j_0)$, where $\mu \in \C$ is a constant.

The same argument as in the half-integral spin case shows that a solution $K \in \Gamma(\Sym^j_0)$ of the above equation satisfies $K \in \ker T^+_j \cap \ker T^-_j$.
Since the space $\ker T^+_j \cap \ker T^-_j \subset \Gamma(\Sym^j_0)$ consists of traceless Killing tensors, see \cite{HMS}, we have the following proposition.
\begin{prop}
  Let $K \in \Gamma(\Sym^j_0)$ be a solution to the differential equation
  \[\nabla_X K = \mu \rho_j(X) K \quad \forall X \in TM.\]
  Then, $K$ is a traceless Killing tensor.
\end{prop}

We remark that Theorem \ref{theo:3dim_Killing_and_Einstein} does not hold in the integral spin case. As a counterexample, we consider $M = \sphere{1} \times \sphere{2}$, which admits a non-zero parallel vector field $V$ arising from the $\sphere{1}$ factor. While $V$ is a solution to the above differential equation for $\mu = 0$, the manifold $\sphere{1} \times \sphere{2}$ is not Einstein. Furthermore, the traceless part of the $j$-th symmetric tensor product of $V$ also provides a solution for $\mu = 0$ for any $j \in \Z_{\geq 0}$.
Indeed, the proof of Theorem \ref{theo:3dim_Killing_and_Einstein} uses essentially the fact that half-integral spin representation has no weight zero.
In the integral spin case, the representation $\pi_j$ contains the zero weight, and thus the same argument fails.
These observations reflect the fact that the Dirac operator $D_j$ is elliptic for half-integral spins, whereas it fails to be elliptic for integral spins.

\subsection{Killing tensors on $\sphere{3}$}
\label{sec:Killing_tensor_on_S3}
In this subsection, we study solutions to the differential equation in \S \ref{sec:Killing_spinor_type_eq_on_int_spin} on the 3-dimensional sphere $\sphere{3}$.
We use the same notations as in \S \ref{sec:higher_killing_on_S3}.

For any element $K \in V_j$, we define a map $K^L \colon G \to V_j$ as $G_L$-invariant and $H$-equivariant.
Since $K^L$ is $H$-equivariant, it defines a section of $\Sym^j_0 = G \times_{H} V_j$ over $\sphere{3} \cong G/H$.
We denote this section by the same symbol $K^L \in \Gamma(\Sym^j_0)$.
Similarly, we can define $K^R \in \Gamma(\Sym^j_0)$ to be $G_R$-invariant and $H$-equivariant.

We write two irreducible components of $\Gamma(\Sym^j_0)$ with highest weights $(0, 2j)$ and $(2j, 0)$ respectively as $K_{2j}^{\pm}$.
By the same argument as Lemma \ref{lemm:irrep_and_higherspin_killing}, we have
\begin{lemm}
  For any $K \in V_j$, $K^L \in K_{2j}^{+}$ and $K^R \in K_{2j}^{-}$ satisfy
  \[\nabla_X K^L = \frac{1}{2}\rho_j(X) K^L, \quad \nabla_X K^R = -\frac{1}{2}\rho_j(X) K^R \quad (\forall X \in \Gamma(T\sphere{3})).\]
  Therefore, $K_{2j}^{+}$ (resp. $K_{2j}^{-}$) is the space of solutions to the differential equation
  \[\nabla_X K = \mu \rho_j(X) K \quad \forall X \in T\sphere{3}\]
  for $\mu = \frac{1}{2}$ (resp. $\mu = -\frac{1}{2}$).
\end{lemm}

On the other hand, according to \cite{HT}, the space of traceless Killing tensors on $\sphere{3}$ is
\[\ker T_j^+ \cap \ker T_j^- = K_{2j}^+ \oplus K_{2j}^-.\]
Thus, we have the following proposition.
\begin{prop}
  Any traceless Killing tensor on $\sphere{3}$ can be decomposed into the sum of two solutions of the differential equation \eqref{eq:int_spin_eq} for $\mu = \frac{1}{2}$ and $\mu = -\frac{1}{2}$.
\end{prop}

\begin{rema}
  It is well-known that all Killing tensors on spheres are obtained as polynomials of Killing vector fields, see \cite{Takeuchi}.
\end{rema}

As we mentioned in Proposition \ref{prop:higher_Killing_spinor_and_Killingtensor}, higher spin Killing spinors give rise to Killing tensors.
We shall discuss this relation in more detail in the case of $\sphere{3}$.
For the purpose of a representation-theoretic treatment, in the following argument, we do not take the real part $\Re$ of the Killing tensors $K^m_{\varphi, \psi}$ and use the complexified tangent bundle $(T\sphere{3})^{\C}$ instead of $T\sphere{3}$.
Namely, we set $K^m_{\varphi, \psi}(X_1, \ldots, X_m) := \langle \pi_j(X_1) \odot\cdots\odot \pi_j(X_m)\varphi, \psi \rangle$ for Killing spinors $\varphi, \psi \in \Gamma(S_j)$ with the same real Killing number and $X_1, \ldots, X_m \in (T\sphere{3})^{\C}$.
Let $K_{2j+1}^+$ be the space of spin $j + \frac{1}{2}$ Killing spinors with Killing number $\mu = \frac{1}{2}$ on $\sphere{3}$ as in \S \ref{sec:higher_killing_on_S3}.
$K_{2j+1}^+$ is isomorphic to $\C \boxtimes W_j$ as an irreducible representation of $G = \SU(2) \times \SU(2)$.
For any $m \in \Z_{\geq 0}$, a linear map
\[K_{2j+1}^{+} \otimes \overline{K_{2j+1}^{+}} \ni \varphi \otimes \psi \mapsto K^m_{\varphi, \psi} \in \Gamma(\Sym^m)\]
is $G$-equivariant, where $\overline{K_{2j+1}^{+}}$ is the complex conjugate representation of $K_{2j+1}^+$.
Let $H, E, F \in \sl(2, \C)$ be the standard basis as in \S \ref{sec:higher_killing_on_H3}, and $\psi_k \in K_{2j+1}^+$ be a weight vector of weight $(0, k)$ for $k = 2j+1, 2j-1, \ldots, -2j-1$.
For simplicity, we write $K^m_{k, l} := K^m_{\psi_k, \psi_l}$ for $k, l = 2j+1, 2j-1, \ldots, -2j-1$.
\begin{lemm}
  \label{lemm:weight_of_K_mkl}
  If a Killing tensor $K^m_{k,l}$ is non-zero, then it is a weight vector of weight $(0, k-l)$.
\end{lemm}
\begin{proof}
  Since $H = -i\sigma_1$, for $(0, H) \in \sl(2, \C) \oplus \sl(2, \C)$, we have
  \begin{align*}
    (0, H) \cdot K^m_{k,l}(X_1, \ldots, X_m) &= \langle \pi_j(X_1) \odot\cdots\odot \pi_j(X_m)((0, H) \cdot \psi_k), \psi_l \rangle - \langle \pi_j(X_1) \odot\cdots\odot \pi_j(X_m)\psi_k, (0, H) \cdot \psi_l \rangle\\
    &= k K^m_{k,l}(X_1, \ldots, X_m) - l K^m_{k,l}(X_1, \ldots, X_m)\\
    &= (k-l) K^m_{k,l}(X_1, \ldots, X_m).
  \end{align*}
  By the same calculation and $G_L$-invariance of $\psi_k, \psi_l$, we have $(H, 0) \cdot K^m_{k,l} = 0$.
\end{proof}
We trivialize the spin structure over $\sphere{3}$ as in \S \ref{sec:higher_killing_on_S3}, which yields $S_j \cong \sphere{3} \times W_j$ and $(T\sphere{3})^{\C} \cong \sphere{3} \times \sl(2,\C)$. 
Under this trivialization, each $\psi_k$ is constant.
We also use the same symbols $E$ and $F$ to denote the global constant vector fields induced by $E, F \in \mathfrak{sl}(2, \mathbb{C})$. Then, for each $k < 2j+1$, there exists a non-zero constant $c \in \mathbb{C}$ such that $\pi_j(E)\psi_k = c \psi_{k+2}$, whereas $\pi_j(E)\psi_{2j+1} = 0$. Similarly, for each $k > -2j-1$, there exists a non-zero constant $c \in \mathbb{C}$ such that $\pi_j(F)\psi_k = c \psi_{k-2}$, whereas $\pi_j(F)\psi_{-2j-1} = 0$.
From this fact, we know that if $2m < |k-l|$, then $K^m_{k,l}$ is zero.

\begin{lemm}
  \label{lemm:hw_condition}
  For each $m = 0, 1, \ldots, 2j+1$, the Killing tensor $K^{m}_{2j+1, 2(j-m)+1}$ satisfies the highest weight condition, that is, it is non-zero and is annihilated by the action of $(0, E) \in \sl(2, \C) \oplus \sl(2, \C)$.
\end{lemm}
\begin{proof}
  By the above argument, we have
    \[K_{2j+1, 2(j-m)+1}^m(F, \ldots, F) = \langle \pi_j(F) \cdots \pi_j(F)\psi_{2j+1}, \psi_{2(j-m)+1} \rangle = c \langle \psi_{2(j-m)+1}, \psi_{2(j-m)+1} \rangle\]
  for some non-zero constant $c \in \C$.
  Thus, $K^m_{2j+1, 2(j-m)+1}$ is non-zero section.
  Next, we show that $(0 ,E) \cdot K^m_{2j+1, 2(j-m)+1} = 0$.
  Since $\psi_{2j+1}$ is a highest weight vector of $K_{2j+1}^+$, we have
  \begin{align*}
    (0, E) \cdot K^m_{2j+1, 2(j-m)+1}(X_1, \ldots, X_m) &= \langle \pi_j(X_1) \odot\cdots\odot \pi_j(X_m)((0, E) \cdot \psi_{2j+1}), \psi_{2(j-m)+1} \rangle\\
    &\hspace{40pt}- \langle \pi_j(X_1) \odot\cdots\odot \pi_j(X_m)\psi_{2j+1}, (0, F) \cdot \psi_{2(j-m)+1} \rangle\\
    &= c \langle \pi_j(X_1) \odot\cdots\odot \pi_j(X_m)\psi_{2j+1}, \psi_{2j+1-2(m+1)} \rangle\\
    &= c K^m_{2j+1, 2(j-(m+1))+1}(X_1, \ldots, X_m)
  \end{align*}
  for some non-zero constant $c \in \C$. Since $2j+1 - (2(j-(m+1))+1) = 2(m+1) > 2m$, $K^m_{2j+1, 2(j-(m+1))+1}$ is zero, and thus we have $(0, E) \cdot K^m_{2j+1, 2(j-m)+1} = 0$ for each $m = 0, 1, \ldots, 2j+1$.
\end{proof}

Let $K^m(\sphere{3}) \subset \Gamma(\Sym^m)$ be the space of Killing tensors of degree $m$ on $\sphere{3}$, which is a $G$-invariant subspace of $\Gamma(\Sym^m)$.
By Lemma \ref{lemm:weight_of_K_mkl} and Lemma \ref{lemm:hw_condition}, each $K^{m}_{2j+1, 2(j-m)+1} \,\, (m = 0, \ldots, 2j+1)$ is a highest weight vector of weight $(0, 2m)$ for some irreducible component in $K^m(\sphere{3})$.
According to \cite{HT}, the space of $m$-th Killing tensors on $\sphere{3}$ is decomposed into irreducible components as
\[K^m(\sphere{3}) \cong \bigoplus_{0 \leq j \leq \lfloor \frac{m}{2} \rfloor} \bigoplus_{0 \leq i \leq \lfloor \frac{m}{2} - j\rfloor} (2m-2i-4j, 2i) \oplus (2i, 2m-2i-4j).\]
There is only one irreducible component with highest weight $(0, 2m)$, and it is the space of traceless Killing tensors $K^+_{2m}$.
Thus, we have the following proposition.
\begin{prop}
  On $\sphere{3}$, for each $m = 0, \ldots, 2j+1$, the Killing tensor $K^m_{2j+1, 2(j-m)+1}$ is a highest weight vector of the irreducible component $K^+_{2m}$ of $\Gamma(\Sym^m_0)$. In particular, $K^m_{2j+1, 2(j-m)+1}$ are traceless Killing tensors.
\end{prop}
A similar argument holds for higher spin Killing spinors with $\mu = -\frac{1}{2}$.
Note that this proposition shows the Clebsch-Gordan formula
\[K^{\pm}_{2j+1} \otimes \overline{K^{\pm}_{2j+1}} \cong K^{\pm}_{2j+1} \otimes K^{\pm}_{2j+1} \cong  K^{\pm}_{4j+2} \oplus K^{\pm}_{4j} \oplus \cdots \oplus K^{\pm}_{2} \oplus K^{\pm}_{0}.\]

\section*{Acknowledgements}
This work was partially supported by JSPS KAKENHI Grant Number JP24K06721, and Waseda University Grants for Special Research Project 2025C-719 and 2025E-013.

\bibliography{refs}

\newpage
\textsc{Yasushi Homma, Department of Mathematics, Faculty of science and engineering, Waseda University, 3-4-1 Ohkubo, Shinjuku-ku, Tokyo 169-8555, Japan.} \\
\vspace{-10pt}

\textit{E-mail address}: \texttt{homma\_yasushi@waseda.jp}\\

\textsc{Natsuki Imada, Department of Pure and applied Mathematics, Graduate school of fundamental science
and engineering, Waseda University, 3-4-1 Ohkubo, Shinjuku-ku, Tokyo, 169-8555, Japan.} \\
\vspace{-10pt}

\textit{E-mail}: \texttt{natsuki.imada@akane.waseda.jp}\\

\textsc{Soma Ohno, Department of Mathematics, Faculty of science and engineering, Waseda University, 3-4-1 Ohkubo, Shinjuku-ku, Tokyo 169-8555, Japan.} \\
\vspace{-10pt}

\textit{E-mail}: \texttt{runhorse@fuji.waseda.jp}

\end{document}